\documentclass[a4paper, reqno]{amsart}

\usepackage{amsthm, amsmath, amssymb}
\usepackage{fullpage}
\usepackage{color}
\usepackage{enumitem}
\PassOptionsToPackage{hyphens}{url}\usepackage{hyperref}
\usepackage[all,cmtip]{xy}
\usepackage[utf8]{inputenc}
\usepackage{multirow}

\setitemize{noitemsep,topsep=0pt,parsep=0pt,partopsep=0pt}
\setenumerate{label=(\alph*), noitemsep,topsep=0pt,parsep=0pt,partopsep=0pt}

\usepackage[T1]{fontenc}
\usepackage{mathptmx}

\theoremstyle{plain}
\newtheorem{theorem}{Theorem}[section]
\newtheorem*{maintheorem}{Main Theorem}
\newtheorem{lemma}[theorem]{Lemma}
\newtheorem{proposition}[theorem]{Proposition}
\newtheorem{corollary}[theorem]{Corollary}

\theoremstyle{definition}

\newtheorem{conjecture}[theorem]{Conjecture}

\theoremstyle{remark}
\newtheorem*{remark}{Remark}

\DeclareMathOperator{\Ker}{Ker}
\DeclareMathOperator{\Coker}{Coker}
\DeclareMathOperator{\Ima}{Im}
\DeclareMathOperator{\Gal}{Gal}
\DeclareMathOperator{\Sel}{Sel}

\DeclareMathOperator{\disc}{disc}
\DeclareMathOperator{\rk}{rank}
\DeclareMathOperator{\ord}{ord}

\newcommand{\calC}{\mathcal{C}}

\newcommand{\frakp}{\mathfrak{p}}
\newcommand{\frakq}{\mathfrak{q}}
\newcommand{\CC}{\mathbf{C}}
\newcommand{\RR}{\mathbf{R}}
\newcommand{\QQ}{\mathbf{Q}}
\newcommand{\ZZ}{\mathbf{Z}}
\newcommand{\FF}{\mathbf{F}}

\newcommand{\sqfr}[1]{#1^\times / #1^{\times 2}}
\newcommand{\torgp}{E(\QQ)_\mathrm{tors}}
\newcommand{\quadsym}[2]{\left( \dfrac{#1}{#2} \right)}

\newcommand{\lp}{\left(}
\newcommand{\rp}{\right)}
\newcommand{\lbr}{\left\{}
\newcommand{\rbr}{\right\}}

\DeclareFontFamily{U}{wncy}{}
\DeclareFontShape{U}{wncy}{m}{n}{<->wncyr10}{}
\DeclareSymbolFont{mcy}{U}{wncy}{m}{n}
\DeclareMathSymbol{\Sh}{\mathord}{mcy}{"58}

\title{On a conjecture of Gross and Zagier}
\author{Dongho Byeon, Taekyung Kim, and Donggeon Yhee}
\date{\today}

\thanks{The first author was supported by Basic Science Research Program through the National Research Foundation of Korea (NRF) funded by the Ministry of Education (NRF-2013R1A1A2007694). The second author was partially supported by Global Ph. D. Fellowship Program through the National Research Foundation of Korea (NRF) funded by the Ministry of Education (grant number 2011-0007588). The third author was supported by Basic Research Program through the National Research Foundation of Korea(NRF) funded by the Ministry of Education (2013053914).}
\subjclass[2010]{11G05}
\keywords{elliptic curves, isogeny, optimal curves, Gross--Zagier conjecture, Birch and Swinnerton-Dyer conjecture, Mordell--Weil Groups, Tamagawa numbers, Manin constant, descent on elliptic curves}

\begin{document}

\begin{abstract}
Let $E$ be an elliptic curve defined over $\QQ$ of conductor $N$, let $M$ be the Manin constant of $E$, and $C$ be the product of local Tamagawa numbers of $E$ at prime divisors of $N$. Let $K$ be an imaginary quadratic field in which each prime divisor of $N$ splits, $P_K$ be the Heegner point in $E(K)$, and $\Sh(E/K)$ be the Tate--Shafarevich group of $E$ over $K$. Also, let $2u_K$ be the number of roots of unity contained in $K$. In \cite{GZ}, Gross and Zagier conjectured that if $P_K$ has infinite order in $E(K)$, then the integer $u_K \cdot C \cdot M \cdot \lp \# \Sh(E/K) \rp^{1/2}$ is divisible by $\#\torgp$. In this paper, we show that this conjecture is true.
\end{abstract}

\maketitle

\tableofcontents

\section{Introduction}

The goal of this paper is to prove a conjecture made by Gross and Zagier in \cite{GZ} concerning certain divisibility among arithmetic invariants of elliptic curves. This gives a theoretical evidence to the ``strong form'' of Birch and Swinnerton-Dyer conjecture, predicting that the leading coefficient of the Hasse--Weil $L$-function of an elliptic curve encodes some precise arithmetic invariants of the curve.

In \cite{GZ}, Gross and Zagier gave a formula for the first derivative at $s=1$ of $L$-series of certain modular forms. In particular, they transferred the formula to the realm of $L$-functions of elliptic curves. So let $E$ be an elliptic curve defined over $\QQ$ with conductor $N$. For a negative square-free integer $d$, we consider the quadratic twist $E_d$ of $E$ which is in general \emph{not} isomorphic to $E$ over $\QQ$ but becomes isomorphic over the imaginary quadratic field $K = \QQ(\sqrt{d})$. We denote the discriminant of $K$ over $\QQ$ by $\disc(K)$ which is equal to $d$ when $d \equiv 1 \pmod{4}$ and to $4d$ otherwise. We also assume a close relation between $E$ and $K$ in such a way that each prime number dividing $N$ splits completely in $K$. This is called the \emph{Heegner condition} or \emph{Heegner hypothesis} in the literature, which we assume throughout this paper. The corresponding $L$-functions are also strongly related: we have $L(E/K, s) = L(E/\QQ,s) \cdot L(E_d/\QQ,s)$. By computing root numbers, the Heegner condition forces that $L(E/K,1)=0$. Throughout this paper, we use the following notations.
\begin{itemize}
    \item $N$ is the conductor of $E$.
    \item $\omega$ is the \emph{Néron differential} of $E$ over $\QQ$ and $\| \omega \|^2 := \int_{E(\CC)} | \omega \wedge \bar{\omega} |$ is the complex period.
    \item $\hat{h}$ is the \emph{Néron--Tate height} attached to $E$.
    \item $M$ is the \emph{Manin constant} of $E$, i.e., if $f$ is the newform attached to $E$ and $\pi: X_0(N) \to E$ is a modular parametrisation, then $M$ is the ratio satisfying $\pi^* \omega = M \cdot 2 \pi i f(\tau) d\tau$. We have $M \in \QQ^\times$ and a famous conjecture of Y. Manin is that $M = 1$ for all \emph{strong Weil curves} $E$. For general discussions on the constant and current status about the conjecture, see \cite{ARS}.
    \item $P_K \in E(K)$ is the \emph{Heegner point} over $K$. This depends on the elliptic curve and its modular parametrisation chosen.
    \item $2u_K$ is the number of roots of unity contained in the field $K$. $u_K = 1$ for all imaginary quadratic fields $K$ except when $K = \QQ(\sqrt{-1})$ and $K = \QQ(\sqrt{-3})$, in these cases we have $u_K = 2$ and $u_K = 3$ respectively.
    \item $C$ is the \emph{Tamagawa number} of $E$ over $\QQ$ which is defined by the product $C = \prod_{p \mid N} C_p$ of all local Tamagawa numbers.
\end{itemize}
Now the main theorem of Gross and Zagier (\cite{GZ}, Theorem I.6.3) has the following consequence.

\begin{theorem}[\cite{GZ}, Theorem V.2.1]
\begin{equation}\label{eq:GZ_formula}
L'(E/K, 1) = \frac{\| \omega \|^2 \cdot \hat{h}(P_K)}{M^2 \cdot u_K^2 \cdot |\disc(K)|^{1/2}}.
\end{equation}
\end{theorem}

Now the Birch and Swinnerton-Dyer conjecture comes into the picture. We assume here and thereafter that the Heegner point $P_K$ has infinite order, so that $L'(E/K,1) \neq 0$. For more details for the following conjecture, we refer \cite{AEC}, appendix C.16.

\begin{conjecture}[Birch and Swinnerton-Dyer]\label{conj:BSD}
If $\ord_{s=1} L(E/K, s) = 1$, then the Tate--Shafarevich group $\Sh(E/K)$ of $E$ over $K$ is finite, and $L'(E/K, s) = \mathrm{BSD}_{E/K}$, where
\begin{equation}\label{eq:BSD_formula}
\mathrm{BSD}_{E/K} = \frac{\| \omega \|^2 \cdot C^2 \cdot \hat{h}(P_K) \cdot \# \Sh(E/K) }{ | \disc(K) |^{1/2} \cdot  \left[ E(K) : \ZZ P_K \right]^2}.
\end{equation}
\end{conjecture}

\begin{remark}
In the literature, the factor $C^2$ in the right hand side of the equation \eqref{eq:BSD_formula} is replaced by the Tamagawa number of $E$ over the extension $K$. However, by the Heegner hypothesis, any prime $p$ dividing $N$ splits in $K$ like $p = \frakp \overline{\frakp}$, and thus the number is equal to the square $C^2$ of the Tamagawa number of $E$ over $\QQ$.
\end{remark}

\begin{remark}
The Tate--Shafarevich group $\Sh(E/K)$ is in fact finite in this case (cf. Theorem 5 in \cite{Kol}).
\end{remark}

Equating the above two formulae \eqref{eq:GZ_formula} and \eqref{eq:BSD_formula}, Gross and Zagier obtained the following conjecture.

\begin{conjecture}[\cite{GZ}, Conjecture V.2.2, Strong Gross--Zagier Conjecture]\label{conj:GZ_strong_conjecture}
If $P_K$ has infinite order in $E(K)$, then $\ZZ P_K$ has finite index in $E(K)$ and we have
\begin{equation}
[E(K): \ZZ P_K ] = u_K \cdot C \cdot M \cdot \lp \# \Sh(E/K) \rp^{1/2}.
\end{equation}
\end{conjecture}

As the order of the rational torsion subgroup $\torgp$ clearly divides the index $\left[ E(K) : \ZZ P_K \right]$, they also obtained a weaker version of the conjecture, which we call ``the Gross--Zagier conjecture'' throughout this paper.

\begin{conjecture}[\cite{GZ}, Conjecture V.2.3, Weak Gross--Zagier Conjecture]\label{conj:GZ_conjecture}
If $E(K)$ has analytic rank 1, then the integer $u_K \cdot C \cdot M  \cdot \lp \# \Sh(E/K) \rp^{1/2}$ is divisible by $\# \torgp$.
\end{conjecture}

Rational torsion subgroups of elliptic curves $E$ over $\QQ$ are completely  classified by Mazur \cite{Ma78}: $\torgp$ is isomorphic to one of the following groups:
\begin{equation*}
\begin{cases}
\ZZ/n\ZZ & \text{ for } 1 \le n \le 10,\; n=12, \\
\ZZ/2\ZZ \oplus \ZZ/n\ZZ& \text{ for } n = 2, 4, 6, 8.
\end{cases}
\end{equation*}
In \cite{Lo}, Lorenzini obtained the following theorem.

\begin{theorem} [\cite{Lo}, Proposition 1.1]
Let $E$ be an elliptic curve defined over $\QQ$ with a $\QQ$-rational point of order $k$. Then the following statements hold with at most five explicit exceptions for a given $k$. The exceptions are given by their labels in Cremona's table \cite{Cr}.
\begin{enumerate}
    \item If $k=4$, then $2 \mid C$, except for `15a7', `15a8', and `17a4'.
    \item If $k=5,6$, or $12$, then $k \mid C$, except for
`11a3', `14a4', `14a6', and `20a2'.
    \item If $k =7,8$, or $9$, then $k^2 \mid C$, except for `15a4', `21a3', `26b1', `42a1', `48a6', `54b3', and `102b1'.
    \item If $k=10$, then $50 \mid C$.
\end{enumerate}
Without exception, $k \mid C$ if $k=7,8,9,10$ or $12$.
\end{theorem}

For the exceptions of above proposition, we can check that $\# \torgp $ divides $C \cdot M$, except for `15a7', which is considered in \S \ref{section:torgp_type_4}. So the only remaining cases for the validity of the conjecture are those when $\torgp$ is isomorphic to the following 6 groups: $\ZZ/2\ZZ$, $\ZZ/3\ZZ$, $\ZZ/4\ZZ$, $\ZZ/2\ZZ \oplus \ZZ/2\ZZ$, $\ZZ/2\ZZ \oplus \ZZ/4\ZZ$, and $\ZZ/2\ZZ \oplus \ZZ/6\ZZ$. Our goal here is to prove these remaining cases, thus to complete the proof of the conjecture.

\begin{maintheorem}
Let $E$ be an elliptic curve defined over $\QQ$ such that the rational torsion subgroup $\torgp$ is isomorphic to one of the 6 groups: $\ZZ/2\ZZ$,
$\ZZ/3\ZZ$, $\ZZ/4\ZZ$, $\ZZ/2\ZZ \oplus \ZZ/2\ZZ$, $\ZZ/2\ZZ \oplus \ZZ/4\ZZ$, and $\ZZ/2\ZZ \oplus \ZZ/6\ZZ$. Let $K$ be an imaginary quadratic field such that $E(K)$ is of (analytic) rank 1 and that $K$ satisfies the Heegner condition. Then the conjecture \ref{conj:GZ_conjecture} is true, i.e., $\#\torgp$ divides $C \cdot M \cdot u_k \cdot \lp \# \Sh(E/K) \rp^{1/2}$.
\end{maintheorem}

From now on, $E$ always denotes an elliptic curve defined over $\QQ$ having torsion subgroup isomorphic to one of the above 6 groups, and $K$ is always an imaginary quadratic field such that $\ord_{s=1} L(E/K, s) = 1$ and that $K$ satisfies the Heegner hypothesis.

Let us briefly explain how to prove the Main Theorem. The present article is divided into two parts. The first part (\S \ref{section:preliminaries_part_1} $\sim$ \S \ref{section:torgp_type_2}) is dealing with the case that $\torgp$ has order a power of 2. When $\torgp$ contains full 2-torsion subgroup $E[2]$, i.e., when $\torgp \simeq \ZZ/2\ZZ \oplus \ZZ/2\ZZ$, or $\ZZ/2\ZZ \oplus \ZZ/4\ZZ$, the situations are a lot easier than the other cases, and we can prove the Main Theorem by computing Tamagawa numbers using Tate's algorithm (\S \ref{section:torgp_type_2_4} and \S \ref{section:torgp_type_2_2}). For the other cases, i.e., when $\torgp \simeq \ZZ/2\ZZ$ or  $\ZZ/4\ZZ$ , there are curves having Tamagawa numbers not divisible by $\# \torgp$, so we need to compute the size of the 2-torsion part of the Tate--Shafarevich groups over $K$ using Kramer's formula. There are some `exceptional families' for which $C \cdot \lp \# \Sh(E/K) \rp^{1/2}$ does not have enough power of 2. For these cases, we avoid difficulties by considering isogeny invariance of the Gross--Zagier conjecture. Kramer's formula and the isogeny invariance are located at the heart of techniques in the proof, so in the preliminary section \S \ref{section:preliminaries_part_1} we give sufficient background to these techniques.

The second part (\S \ref{section:preliminaries_part_2} $\sim$ \S \ref{section:torgp_type_3}) is devoted to the case in which $\torgp$ has a rational torsion point of order 3. When $\torgp \simeq \ZZ/2\ZZ \oplus \ZZ/6\ZZ$, we can prove the Main Theorem by computing only Tamagawa numbers (\S \ref{section:torgp_type_2_6}). But when $\torgp \simeq \ZZ/3\ZZ$, there are also curves having Tamagawa numbers not divisible by $\# \torgp$, so we need to compute the lower bound of the size of the 3-torsion part of Tate--Shafarevich groups over $K$ using Cassels' formula or need to compute the Manin constants using the phenomenon that optimal curves differ by a 3-isogeny (\S \ref{section:torgp_type_3}). Same as Part 1, preliminaries are summarised in \S \ref{section:preliminaries_part_2}.

All expicit computations in this paper were done using Sage Mathematics Software \cite{sagemath}. When we do computations with Weierstrass equations, we frequently change the variables of an equation to obtain another. In particular, when we use the clause ``make a change of variables via $[u,r,s,t]$'', it should be understood to take the change of variables formula given by
\begin{equation*}
x = u^2 x' + r \quad \text{and} \quad y = u^3 y' + u^2 s x'+ t.
\end{equation*}
For the details, we refer \cite{AEC}, \S III.1.

\part{$E(\QQ)_\mathrm{tors}$ has order a power of 2}

\section{Preliminaries for Part 1}
\label{section:preliminaries_part_1}

\subsection{Kramer's formula}
\label{subsection:Kramer}

In this subsection we introduce a formula of Kramer \cite{Kr}, and discuss how to measure the size of the Tate--Shafarevich group of elliptic curve using it. Of course the purpose of this section is to provide a tool to show the Main Theorem for the cases $\torgp \simeq \ZZ/2\ZZ$ or $\ZZ/4\ZZ$. Thus, throughout this subsection, we assume $\torgp \simeq \ZZ/2\ZZ$ or $\ZZ/4\ZZ$, and consequently $E(\QQ)[2] \simeq \ZZ/2\ZZ$.

Since the Tate--Shafarevich group $\Sh(E/K)$ is finite (Theorem 5 in \cite{Kol}), its 2-primary part $\Sh(E/K)[2^\infty]$ has perfect square order. So if we find a non-trivial element in $\Sh(E/K)[2]$, (or equivalently $\dim_{\FF_2} \Sh(E/K)[2] \ge 1$), we can immediately see that $2 \mid \lp \# \Sh(E/K) \rp^{1/2}$. So in this subsection, we are concentrating on how to find such a non-trivial element.

Let $p$ be a prime number. We use the following notations.
\begin{itemize}
    \item $i_p = \dim_{\FF_2} \Coker N = \dim_{\FF_2} E(\QQ_p)/NE(K_\frakp)$, where $N: E(K_\frakp) \to E(\QQ_p)$ is the \emph{norm map}. This quantity is called \emph{local norm index} of $E$ at $p$.
    \item Let
    \begin{equation*}
    \Phi = \left\{ \xi \in \Sel^2(E/\QQ) : \xi \in N_p\left( \prod_{\frakp \mid p} \Sel^2(E/K_\frakp) \right) \right\}.
    \end{equation*}
    This group is called the \emph{everywhere-local norm group}.
    \item $NS'$ is the image of the norm map $\Sel^2(E/K) \to \Sel^2(E/\QQ)$, which we do not need in this paper.
\end{itemize}

\begin{theorem}[\cite{Kr}, Theorem 1] \label{th:Kr}
The dimension of $\Sh (E/K)[2]$ (over $\FF_2$) is equal to
\begin{equation*}
\sum i_\ell + \dim_{\FF_2} \Phi + \dim_{\FF_2} NS' - \rk E(K) - 2 \dim_{\FF_2} E(\QQ)[2],
\end{equation*}
where the sum is taken over all primes (including infinity) of $\QQ$.
\end{theorem}

Back to our case. Because $\rk E(K)=1$ and $E(\QQ)[2] \simeq \ZZ/2\ZZ$, by Theorem \ref{th:Kr}, $\dim_{\FF_2} \Sh(E/K)[2] \ge 1$ if and only if the quantity
\begin{equation*}
\sum i_\ell + \dim_{\FF_2} \Phi + \dim_{\FF_2} NS'
\end{equation*}
is greater than or equal to 4.

\subsubsection{Local norm indices}

For general introduction and useful facts about the numbers $i_p$, we refer \S 4 of \cite{Ma72} and \S 2 of \cite{Kr}. We only concern those numbers relevant to our situation. The proof of the following proposition can be found in \S 2 of \cite{Kr}.

\begin{proposition}\label{prop:computing_local_norm_indices}
Let $E$ be an elliptic curve over $\QQ$ with $E(\QQ)[2] \simeq \ZZ/2\ZZ$ and let $K =\QQ(\sqrt{d})$ be an imaginary quadratic field satisfying the Heegner hypothesis. The local norm indices $i_\ell$ for various primes $\ell$ are given as follows.
\begin{enumerate}
    \item One has
$i_\infty = i(\CC\mid\RR) = \begin{cases}
                           0 & \text{ if } \Delta_\mathrm{min} < 0, \\
                           1 & \text{ if } \Delta_\mathrm{min} > 0.
                           \end{cases}$
    \item Let $p$ be an odd prime. If $p$ is a good prime for $E$ and is ramified in $K$, then one has $i_p = \dim_{\FF_2} E[2](k)$, where $k$ is the residue field of $\QQ_p$. Otherwise one has $i_p = 0$.
    \item If $2$ is a good prime for $E$ and is ramified in $K$, then one has $i_2 = \begin{cases}
2 & \text{if $(\Delta_\mathrm{min},d)_{\QQ_2} = +1$,} \\
1 & \text{if $(\Delta_\mathrm{min},d)_{\QQ_2} = -1$,}
\end{cases}$ where $(-,-)_{\QQ_2}$ denotes the Hilbert norm-residue symbol. Otherwise, one has $i_2 =0$.
\end{enumerate}
\end{proposition}

\subsubsection{Everywhere-local norm group}

Now we provide a way to compute the everywhere-local norm group $\Phi$. The following is the key.

\begin{proposition}[\cite{Kr}, Proposition 7]\label{prop:str_of_Phi}
The everywhere-local norm group $\Phi$ is the intersection of $\Sel^2(E/\QQ)$ and $\Sel^2(E_d/\QQ)$ inside $H^1(\QQ, E[2]) \simeq H^1(\QQ, E_d[2])$, where $E_d$ is the quadratic twist of $E$ by $d$.
\end{proposition}

Let $E_d$ be the quadratic twist of $E$ by $d$. In particular, suppose $E$ is defined by the Weierstrass equation
\begin{equation}\label{eq:Weierstrass_having_2_torsion}
y^2 = x^3 + Ax^2 + Bx,
\end{equation}
which has discriminant $\Delta = 2^4 B^2 ( A^2 - 4B)$. Then $E_d$ has the Weierstrass equation of the form
\begin{equation}\label{eq:Weierstrass_quadratic_twist}
y^2 = x^3 + Ad x^2 + Bd^2x.
\end{equation}
The discriminant of the above equation \eqref{eq:Weierstrass_quadratic_twist} is given by $\Delta_d = 16d^6 B^2(A^2 - 4B)$.

\begin{proposition}\label{prop:canonical_isomorphism_of_torsion_groups}
The 2-torsion subgroups $E[2]$ and $E_d[2]$ are canonically isomorphic as $\Gal(\overline{\QQ}|\QQ)$-modules. Consequently, the Galois cohomology groups $H^\bullet (\QQ, E[2])$ and $H^\bullet(\QQ, E_d[2])$ are isomorphic. In particular, we identify $H^1(\QQ, E[2]) = H^1(\QQ, E_d[2])$ in the sequel.
\end{proposition}

\begin{proof}
The Galois-equivariant isomorphism $E[2] \to E_d[2]$ is given by $(t,0) \mapsto (dt,0)$.
\end{proof}

Denote by $P$ (resp. $P_d$) the rational torsion point of order 2 in $E$ (resp. $E_d$) corresponding to $(0,0)$ in the equation \eqref{eq:Weierstrass_having_2_torsion} (resp. $(0,0)$ in the equation \eqref{eq:Weierstrass_quadratic_twist}). Let $E'$ (resp. $E'_d$) be the elliptic curve $E/\langle P \rangle$ (resp. $E_d / \langle P_d \rangle$) and let $\phi$ (resp. $\phi_d$) be the canonical quotient 2-isogeny $E \to E'$ (resp. $E_d \to E'_d$).

\begin{proposition}
There are canonical homomorphisms
\begin{equation*}
H^1(\QQ, E[\phi]) \to H^1(\QQ, E[2]), \qquad \text{and} \qquad
H^1(\QQ, E_d[\phi_d]) \to H^1(\QQ, E_d[2]),
\end{equation*}
and they induce
\begin{equation*}
\Sel^\phi (E/\QQ) \to \Sel^2 (E/\QQ), \qquad \text{and} \qquad
\Sel^{\phi_d} (E_d/\QQ) \to \Sel^2 (E_d/\QQ).
\end{equation*}
\end{proposition}

\begin{proof}
If we denote the unique dual rational 2-isogeny of $\phi$ by $\phi'$, then we have a canonical exact sequence
\begin{equation}\label{eq:isogeny_SES}
0 \longrightarrow E[\phi] \longrightarrow E[2] \longrightarrow E'[\phi'] \longrightarrow 0.
\end{equation}
This defines a canonical map $H^1(\QQ, E[\phi]) \to H^1(\QQ, E[2])$ on cohomology groups, and it restricts to the map $\Sel^\phi (E/\QQ) \to \Sel^2 (E/\QQ)$ of subgroups. For $E_d$ and $\phi_d$ the proof is \textit{mutatis mutandis} the same.
\end{proof}

\begin{proposition}
There are canoncial isomorphisms $H^1(\QQ, E[\phi]) \simeq \sqfr{\QQ}$ and $H^1(\QQ, E_d[\phi_d]) \simeq \sqfr{\QQ}$. Moreover, the isomorphisms are compatible in the sense that the following diagram is commutative:
\begin{equation*}\begin{gathered}
\xymatrix{
& H^1(\QQ, E[\phi]) \ar[dd] \ar[r] & H^1(\QQ,E[2]) \ar[dd]^= \\
\sqfr{\QQ} \simeq H^1(\QQ,\mu_2) \ar[ur]^\sim \ar[dr]_\sim && \\
& H^1(\QQ, E_d[\phi_d]) \ar[r] & H^1(\QQ,E_d[2]) \\
}
\end{gathered}\end{equation*}
where the vertical map in the middle is induced by the canonical isomorphism in the Proposition \ref{prop:canonical_isomorphism_of_torsion_groups}.
\end{proposition}

\begin{proof}
Clearly the isomorphisms $\mu_2 \to E[\phi]$ and $\mu_2 \to E_d[\phi_d]$ are compatible in the sense the left triangle commutes. By Kummer theory we know $H^1(\QQ, \mu_2) =\sqfr{\QQ}$, whence the result follows. 
\end{proof}

\begin{proposition}\label{prop:kernel_of_selmer}
Let $G$ be the subgroup of $\sqfr{\QQ}$ generated by the class of $A^2 - 4B$. Then $G$ is the kernel of the homomorphisms $H^1(\QQ, E[\phi]) \to H^1(\QQ, E[2])$ and $H^1(\QQ, E_d[\phi_d]) \to H^1(\QQ, E_d[2])$. Thus,
\begin{equation*}
\Ker \lp \Sel^\phi (E/\QQ) \to \Sel^2 (E/\QQ) \rp = G \cap \Sel^\phi(E/\QQ) \subset \Sel^\phi (E/\QQ).
\end{equation*}
Similarly,
\begin{equation*}
\Ker \lp \Sel^{\phi_d} (E_d/\QQ) \to \Sel^2 (E_d/\QQ) \rp = G \cap \Sel^{\phi_d}(E_d/\QQ) \subset \Sel^\phi (E_d/\QQ).
\end{equation*}
\end{proposition}

\begin{proof}
We only give a proof for $E$ and $\phi$. For $E_d$ and $\phi_d$, everything is the same under making certain notational change. From the short exact sequence \eqref{eq:isogeny_SES}, we have the long exact sequence of cohomology groups:
\begin{equation*}
0 \to E(\QQ)[\phi] \to E(\QQ)[2] \to E'(\QQ)[\phi'] \xrightarrow{\eta} H^1(\QQ,E[\phi]) \to H^1(\QQ, E[2]) \to H^1(\QQ,E'[\phi']) \to \cdots
\end{equation*}
Because we only consider those elliptic curves with $E(\QQ)[\phi] = E(\QQ)[2]$, the map $E(\QQ)[2] \to E'(\QQ)[\phi']$ is the zero map, and this again forces us that $\eta : E'(\QQ)[\phi'] \to H^1(\QQ,E[\phi])$ is injective. The image $\eta \lp E'(\QQ)[\phi'] \rp$ is the kernel of $H^1(\QQ, E[\phi]) \to H^1(\QQ, E[2])$.

We claim that this kernel is equal to $G$. Write $E(\overline{\QQ})[2] = \lbrace O, P, Q, P+Q \rbrace$, where $O$ is the identity of $E$ and $P \in E(\QQ)$, and similarly write $E'(\overline{\QQ})[\phi'] = \lbrace O', T \rbrace$, where $O'$ is the identity of $E'$. Clearly $T \in E'(\QQ)$. Since $E(\overline{\QQ})[2] \to E'(\overline{\QQ})[\phi']$ is surjective but $E(\QQ)[2] \to E'(\QQ)[\phi']$ is the zero map, the point $Q$ is mapped onto $T$ under $E(\overline{\QQ})[2] \to E'(\overline{\QQ})[\phi']$. Then, $\eta(T) \in H^1(\QQ,E[\phi])$ is defined by the 1-cocyle
\begin{equation*}
\sigma \mapsto \sigma(Q) - Q = \begin{cases}
P & \text{ if $\sigma(Q) = P+Q \neq Q$, } \\
0 & \text{ if $\sigma(Q) = Q$.}
\end{cases}
\end{equation*}
However, this 1-cocycle corresponds to the 1-cocycle $\sigma \mapsto \sigma(\sqrt{b})/\sqrt{b}$ defining an element $H^1(\QQ, \mu_2)$, where $b=A^2 - 4B$, since in the Weierstrass equation \eqref{eq:Weierstrass_having_2_torsion}, $Q$ corresponds to the point $\displaystyle \left( \frac{-A \pm \sqrt{A^2 - 4B}}{2}, 0 \right)$ and thus $\sigma(Q) = Q$ if and only if $\sigma \left( \sqrt{A^2 - 4B} \right) = \sqrt{A^2 - 4B}$. Clearly the 1-cocycle $\sigma \mapsto \dfrac{\sigma(\sqrt{A^2-4B})}{\sqrt{A^2 - 4B}}$ defining an element $H^1(\QQ, \mu_2)$ corresponds to $A^2 - 4B$ in $\sqfr{\QQ}$.
\end{proof}

Recall (Proposition \ref{prop:str_of_Phi}) that the everywhere-local norm group $\Phi$ is the intersection of two Selmer groups $\Sel^2(E/\QQ)$ and $\Sel^2(E_d/\QQ)$ inside $H^1(\QQ, E[2]) = H^1(\QQ, E_d[2])$. In order to identify elements in the intersection, we need to find $b \in \sqfr{\QQ}$ such that $b \in \Sel^\phi(E/\QQ) \cap \Sel^{\phi_d}(E/\QQ)$ by descent arguments (cf. \cite{AEC}, chapter X). In order to ensure this is not the identity element in $\Phi$, we should check $b \not\in G$. This will be done when we deal with $\torgp \simeq \ZZ/4\ZZ$ or $\ZZ/2\ZZ$.

\subsection{Isogeny invariance of the Gross--Zagier conjecture}

Let $E$ and $E'$ be isogenous elliptic curves defined over $\QQ$, and $K$ be an imaginary quadratic field satisfying the Heegner hypothesis. We consider those curves with fixed modular parametrisations $\pi: X_0(N) \to E$ and $\pi': X_0(N) \to E'$.

\begin{proposition}\label{prop:isoginv}
Let $\theta: E \to E'$ be a rational isogeny.
\begin{enumerate}
\item
If the strong Gross--Zagier conjecture (Conjecture \ref{conj:GZ_strong_conjecture}) is true for $E$ then it is also true for $E'$.

\item
Suppose that $\theta$ respects modular parametrisations of $E$ and $E'$, i.e., $\pi' = \theta \circ \pi$. Then we have
\begin{equation}\label{eq:isoginv}
\frac{M^2 \cdot C^2 \cdot \# \Sh(E/K)}{[E(K):\ZZ P_K]^2} = \frac{M'^2 \cdot C'^2 \cdot \# \Sh(E'/K)}{[E'(K):\ZZ P_K']^2}.
\end{equation}

\item
Let $p$ be a prime. If
\begin{enumerate}[label=(\roman*)]
\item
$\ord_p \# E(K)_\mathrm{tors} = \ord_p \# E(\QQ)_\mathrm{tors}$, and

\item
$\ord_p \# E(\QQ)_\mathrm{tors} \le \ord_p \lp u_K \cdot C \cdot M \cdot \lp \# \Sh(E/K) \rp^{1/2} \rp$,
\end{enumerate}
then
\begin{equation*}
\ord_p \# E'(\QQ)_\mathrm{tors} \le \ord_p \lp u_K \cdot C' \cdot M' \cdot \lp \# \Sh(E'/K) \rp^{1/2} \rp.
\end{equation*}
In particular, if $E(K)_\mathrm{tors} = E(\QQ)_\mathrm{tors}$, and if the weak Gross--Zagier conjecture (Conjecture \ref{conj:GZ_conjecture}) for $E$ is true, then it is also true for $E'$.
\end{enumerate}
\end{proposition}

\begin{proof}
(a) Isogenous curves $E$ and $E'$ have the same $L$-functions and the same BSD formulae, i.e., $L(E/K, s) = L(E'/K, s)$ and $\mathrm{BSD}_{E/K} = \mathrm{BSD}_{E'/K}$ (cf. Conjecture \ref{conj:BSD}). The latter is a theorem of Cassels \cite{Ca1}. As the strong Gross--Zagier conjecture is obtained by simply equating these formulae, it is clearly isogeny invariant.

(b) Let $P_K'$ be the Heegner point for $E'$ defined by $P_K' = \theta (P_K)$. Since $L'(E/K, s) = L'(E'/K,s)$, we have
\begin{equation*}
\frac{\| \omega \|^2 \cdot \hat{h}(P_K)}{\| \omega' \|^2 \cdot \hat{h}(P_K')} = \frac{M^2}{M'^2}.
\end{equation*}
Similarly, from $\mathrm{BSD}_{E/K} = \mathrm{BSD}_{E'/K}$, we get
\begin{equation*}
\frac{\| \omega \|^2 \cdot \hat{h}(P_K)}{\| \omega' \|^2 \cdot \hat{h}(P_K')} = \frac{\# \Sh(E'/K) \cdot C'^2 \cdot [E(K):\ZZ P_K]^2}{\# \Sh(E/K) \cdot C^2 \cdot [E'(K):\ZZ P_K']^2}.
\end{equation*}
Equating, we obtain the equation \ref{eq:isoginv}.

(c) Let $P$ (resp. $P'$) be a generator of the group $E(K)/E(K)_\mathrm{tors}$ (resp. $E'(K)/E'(K)_\mathrm{tors}$), and let $P_K = \nu P$ (resp. $P_K' = \nu' P'$). As $P_K' = \theta(P_K) = \nu \theta(P)$, the index $\nu'$ is divisible by $\nu$. The assumption (i) $\ord_p \# E(K)_\mathrm{tors} = \ord_p \# E(\QQ)_\mathrm{tors}$ implies that $\ord_p [E(K)_\mathrm{tors}: E(\QQ)_\mathrm{tors}] = 0$. By the equation \ref{eq:isoginv}, we have
\begin{equation*}
\frac{u_K^2 \cdot M^2 \cdot C^2 \cdot \# \Sh(E/K)}{\lp \# E(\QQ)_\mathrm{tors} \rp^2 \cdot [E(K)_\mathrm{tors}:E(\QQ)_\mathrm{tors}]^2} = \frac{u_K^2 \cdot M'^2 \cdot C'^2 \cdot \# \Sh(E'/K)}{\lp \frac{\nu'}{\nu} \rp^2 \cdot \lp \# E'(\QQ)_\mathrm{tors} \rp^2 \cdot [E'(K)_\mathrm{tors}:E'(\QQ)_\mathrm{tors}]^2},
\end{equation*}
and by the assumption (ii) the left hand side of the above equation is a $p$-adic integer. Thus,
\begin{align*}
\ord_p \lp u_K \cdot C' \cdot M' \cdot \lp \# \Sh(E'/K) \rp^{1/2} \rp & \ge \ord_p \lp \frac{\nu'}{\nu}  \cdot \lp \# E'(\QQ)_\mathrm{tors} \rp \cdot [E'(K)_\mathrm{tors}:E'(\QQ)_\mathrm{tors}] \rp \\
& \ge \ord_p \# E'(\QQ)_\mathrm{tors}.
\end{align*}
\end{proof}

\begin{remark}
By \cite{GJT} Corollary 4 or \cite{Naj}, Theorem 2, for a given elliptic curve $E$ defined over $\QQ$, there are at most 4 quadratic fields $K$ such that $E(K)_\mathrm{tors} \neq E(\QQ)_\mathrm{tors}$.
\end{remark}

\section{$\torgp \simeq \ZZ/2\ZZ \oplus \ZZ/4\ZZ$}
\label{section:torgp_type_2_4}

In this section, we prove the Main Theorem for the cases when $\torgp$ is isomorphic to $\ZZ/2\ZZ \oplus \ZZ/4\ZZ$.

\begin{theorem}\label{th:torgp_type_2_4}
Suppose that $\torgp$ is isomorphic to $\ZZ/2\ZZ \oplus \ZZ/4\ZZ$. Then the order $8 = \# \torgp$ divides the Tamagawa number $C$ of $E$, except for the curve `15a3', in which case $C \cdot M = 8$.
\end{theorem}

From \cite{Ku}, table 3, such elliptic curves can be parametrised by one parameter $\lambda \in \QQ$ by
\begin{equation}\label{eq:Weierstrass_basic_for_type_2_4}
y^2 + xy - \lambda y = x^3 - \lambda x^2,
\end{equation}
where $\displaystyle \lambda = \left( \frac{\alpha}{\beta} \right)^2 - \frac{1}{16} = \frac{16\alpha^2 - \beta^2}{16\beta^2}$, with positive integers $\alpha, \beta$ having no common prime divisor, and $\alpha/\beta \neq  1/4$. The discriminant of the equation is $\Delta = \lambda^4(1+16\lambda) \neq 0$. Note that since we take $\alpha$ and $\beta$ relatively prime, there are no common prime divisor of $16\alpha^2 - \beta^2$ and $16\beta^2$ except 2.

\begin{proposition}
\label{prop:something_good_happened_for_eq_lambda}
Let $p$ be a prime.

\begin{enumerate}
\item
If $m:=\ord_p \lambda >0$, then the reduction of $E$ modulo $p$ is (split) multiplicative of type $\mathrm{I}_{4m}$. Consequently the Tamagawa number at $p$  of $E$ is $C_p = 4m$.

\item
Suppose that $p \neq 2$. If $m := \ord_p \lambda < 0$, then $m$ is always even, and the minimal Weierstrass equation at $p$ is given by
\begin{equation}\label{eq:m-even}
y^2 + p^z xy -up^z y = x^3 -ux^2,
\end{equation}
where $u \in \ZZ_p^\times$ satisfying $\lambda = up^m$ in $\ZZ_p$, and where $z$ is a positive integer. The reduction type of the equation modulo $p$ is $\mathrm{I}_n$ with $n = 2z$, whence $C_p = 2z$.
\end{enumerate}
\end{proposition}

\begin{proof}
(a) This can be shown by directly applying Tate's algorithm (see \cite{advAEC}, \S IV.9) to the Weierstrass equation \eqref{eq:Weierstrass_basic_for_type_2_4}.

(b) Since $\gcd \lp 16\alpha^2 - \beta^2, 16\beta^2 \rp$ is a power of $2$, if $m = \ord_p \lambda  = \ord_p \lp  (16\alpha^2 - \beta^2) / 16\beta^2 \rp <0$ then the exponent $m$ is always even. Changing Weierstrass equation (cf. \cite{Lo}, proof of Proposition 2.4), we get the equation \eqref{eq:m-even}. We use Tate's algorithm again for this equation to obtain the minimality and reduction type.
\end{proof}

Let
\begin{equation*}
S = \lbr p \text{ primes}: \ord_p \lambda > 0 \rbr, \qquad
T = \lbr p \text{ primes}: p \neq 2,\, \ord_p \lambda < 0 \rbr.
\end{equation*}
Proposition \ref{prop:something_good_happened_for_eq_lambda} says that Theorem \ref{th:torgp_type_2_4} is true
when (i) $\# S \ge 2$; or (ii) $\# S = 1$ and $\# T \ge 1$.
Thus the following proposition shows Theorem \ref{th:torgp_type_2_4}.


\begin{proposition}\label{prop:relations_of_S_and_T}
With possible finite number of exceptions, we have $\# S \ge 1$ and moreover if $T = \emptyset$, then $\# S \ge 2$. The exceptions are exactly the following curves: `15a1', `15a3', `21a1', `24a1', `48a3', `120a2', `240a3', `240d5', and `336e4'. But in any case including these exceptions, we have $8 \mid C \cdot M$.
\end{proposition}

\begin{proof}
Write $\beta = 2^n \beta'$ with $n \ge 0$ and $\beta'$ odd. The condition $T = \emptyset$ is equivalent to the condition $\beta' = 1$. We divide the proof according to the value of $n$.

Suppose $n=0$.
In this case $16\alpha^2 - \beta^2$ is odd and so $\gcd \lp 16\alpha^2 - \beta^2, 16\beta^2 \rp =1$.
Suppose that there is no prime dividing $16\alpha^2 - \beta^2$. We then have $16\alpha^2 - \beta^2 = \pm 1$.
This is possible only if $\alpha = 0$, a contradiction. For the second statement, assume $\beta = 1$.
In this case, we get $\lambda = (16\alpha^2 - 1)/16$.
If there is only one odd prime $p$ dividing $16\alpha^2 - 1$,
then we must have $4\alpha - 1 =1$, a contradiction ($\alpha \in \ZZ_{>0}$).

Suppose $n=1$. We have $\displaystyle \lambda = \frac{16\alpha^2 -
4\beta'^2}{16\cdot 4 \beta'^2}
 = \frac{4\alpha^2 - \beta'^2}{16\beta'^2}$, and $\gcd \lp 4\alpha^2 - \beta'^2, 16\beta'^2 \rp =1$.
 If there were no odd prime dividing $4\alpha^2 - \beta'^2$, we would have $4\alpha^2 - \beta'^2 = \pm 1$,
 whence $\alpha = 0$, a contradiction. If $\beta=2$ (equivalently $\beta' = 1$),
 and if there were only one prime dividing $4\alpha^2 - \beta'^2$,
 then either one of the relations $2\alpha - 1 =1$ or $2\alpha + 1 = -1$ would hold.
 Thus we must have $\alpha = 1$. In this case we get the curve `48a3', having $C_2 = C_3 = 4$.

Suppose $n \ge 5$. In this case we can take another Weierstrass
equation (cf. Equation \eqref{eq:m-even}) of $E$ of the following
form:
\begin{equation}\label{eq:at2}
y^2 + 2^n xy - 2^n u y = x^3 - ux^2,
\end{equation}
where $u = \alpha^2 - 2^{2n-4}$. This equation has discriminant
$\Delta = (2^{2n} + 16u) 2^{2n} u^4$ and $c_4 = 2^{2n + 4} u + 16
u^2 + 2^{4n}$, so $\ord_2(\Delta) = 2n+4$ and $\ord_2(c_4) = 4$.
Moreover, by \cite{AEC}, Proposition VII.5.5, since its
$j$-invariant has order $8-2n < 0$, $E$ has potentially
multiplicative reduction modulo $2$. If this equation \eqref{eq:at2}
is minimal at the prime 2, then the curve has additive reduction
modulo 2 ($\ord_2(c_4)>0$). Tate's algorithm says that $E$ has
reduction of type $\mathrm{I}_k^\ast$ for some $k$, with Tamagawa
number 2 or 4. Suppose that the equation \eqref{eq:at2} is not
minimal modulo $2$. Then we can transform \eqref{eq:at2} into a
minimal model modulo $2$, which has discriminant of order $2n+4 - 12
= 2n - 8$ at 2 and $c_4$ of order 0. Since the order of the minimal
discriminant is even and $>0$, and since $E$ has multiplicative
reduction ($\ord_2 c_4 = 0$), we have even $C_2$ by Tate's
algorithm. As $C_2$ is even and $4 \mid C_p$ for some odd $p \in S$,
the proof of this case is completed.

Remaining cases ($n = 2, 3,$ and 4) can be shown similarly.
\end{proof}

\section{$\torgp \simeq \ZZ/2\ZZ \oplus \ZZ/2\ZZ$}
\label{section:torgp_type_2_2}

In this section, we prove the Main Theorem for the cases when $\torgp$ is isomorphic to $\ZZ/2\ZZ \oplus \ZZ/2\ZZ$.

\begin{theorem}\label{th:torgp_type_2_2}
Suppose that $\torgp$ is isomorphic to $\ZZ/2\ZZ \oplus \ZZ/2\ZZ$. Then the order $4 = \# \torgp$ divides the Tamagawa number $C$ of $E$, except for two curves `17a2' and `32a2'. For these two cases we have $4 = C \cdot M$.
\end{theorem}

Following \cite{Ku}, we can take a Weierstrass model of the form
\begin{equation}\label{eq:Weierstrass_basic_for_type_2_2}
y^2 = x(x+a)(x+b),
\end{equation}
where $a, b \in \ZZ$ with $a \neq b \neq 0 \neq a$. Note that $a$ and $b$ is in general not relatively prime. The discriminant of the equation \eqref{eq:Weierstrass_basic_for_type_2_2} is $\Delta = 16 (a - b)^2 a^2 b^2$ and $c_4 = 16a^2 - 16 ab + 16 b^2$. Let $c = a-b \neq 0$. If there is a prime $p$ dividing both $a$ and $b$, then by changing the equation via $[p, 0, 0, 0]$ if necessary, we assume $\min \lp \ord_p a, \ord_p b \rp = 1$.

We first investigate the Tamagawa number $C_p$ for primes $p$ dividing $abc$.

\begin{proposition}\label{prop:prime_dividing_a_xor_b}
Let $p$ be a prime. Assume that either (i) $p \mid a$ and $p \nmid bc$; or (ii) $p \mid b$ and $p \nmid ac$. Then we have the following.
\begin{enumerate}
    \item If $p$ is odd, then $E$ has reduction of type $\mathrm{I}_{\ord_p \Delta} = \mathrm{I}_{2\ord_p(a)}$ modulo $p$, with even Tamagawa number at $p$.
    \item Suppose that $p = 2$. If $m:=\ord_2 a = 4$ and if $b \equiv 1 \pmod{4}$, then $E$ has good reduction modulo $2$, whence $C_2 = 1$. Otherwise, $C_2$ is even.
\end{enumerate}
\end{proposition}

\begin{proof}
We only give the proof for the case (i). By the symmetry of the roles of $a$ and $b$ in the equation, the case (ii) follows immediately.

(a) This is immediate from Tate's algorithm.

(b) Suppose that $2 \mid a$ and $2 \nmid bc$. We do a case-by-case study. In order to help readers to re-construct proofs of the results in the following table, we remark that we mostly apply Tate's algorithm to the Weierstrass equation \eqref{eq:Weierstrass_basic_for_type_2_2}, while for the case $m=4$ and $b \equiv 1 \pmod{4}$ and for $m \ge 5$, we apply the algorithm to another Weierstrass equation $\displaystyle y^2 + xy = x^3 + \frac{a+b-1}{4}x^2 + \frac{ab}{2^4} x$.

\begin{center}
\begin{tabular}{|c|c|c|c|}
\hline
$m$ & $b \mod{4}$ & Reduction Type of $E$ at $p=2$ & $C_2$ \\ \hline\hline
$1$ & $1$ or $3$ & $\mathrm{III}$ & 2 \\ \hline
\multirow{2}{*}{$2$} & $1$ & $\mathrm{I}_n^\ast$ & 2 or 4 \\ \cline{2-4}
 & $3$ & $\mathrm{I}_0^\ast$ & 2 \\ \hline
\multirow{2}{*}{$3$} & $1$ & $\mathrm{III}^\ast$ & 2 \\ \cline{2-4}
 & $3$ & $\mathrm{I}_n^\ast$ & 2 or 4 \\ \hline
\multirow{2}{*}{$4$} & $1$ & $\mathrm{I}_0$ (good) & 1 \\ \cline{2-4}
 & $3$ & $\mathrm{I}_n^\ast$ & 2 or 4 \\ \hline
$ \ge 5$ & $1$ or $3$ & $\mathrm{I}_{2m - 8}$ & even \\ \hline
\end{tabular}
\end{center}
\end{proof}

\begin{proposition}\label{prop:prime_dividing_c}
Let $p$ be a prime such that $p \mid c$ and $p \nmid ab$.
\begin{enumerate}
    \item If $p$ is odd, then $E$ has reduction of type $\mathrm{I}_{\ord_p \Delta} = \mathrm{I}_{2\ord_p(c)}$ modulo $p$, with even Tamagawa number at $p$.
    \item Suppose that $p = 2$. If $m:=\ord_2 c = 4$ and if $a \equiv b \equiv 3 \pmod{4}$, then $E$ has good reduction modulo $2$, whence $C_2 = 1$. Otherwise, $C_2$ is even.
\end{enumerate}
\end{proposition}

\begin{proof}
We make a change of variables via $[1,-a,0,0]$, to get another equation
\begin{equation}\label{eq:for_looking_c}
y^2 = x^3 + (-2a+b)x^2 + a(a-b)x.
\end{equation}

(a) Immediate from Tate's algorithm applied to equation \eqref{eq:for_looking_c}.

(b) Let $p=2$. Similar as above proposition, the results from Tate's algorithm applied to the equation \eqref{eq:for_looking_c} are summarised as follows. In particular, when dealing with the cases $a \equiv b \equiv 3 \pmod{4}$ and $m \ge 4$, we use the equation $\displaystyle y^2 + xy = x^3 + \frac{-2c -b -1}{4} x^2 + \frac{ac}{16} x$ instead.

\begin{center}
\begin{tabular}{|c|c|c|c|}
\hline
$m$            & $a$ and $b \mod{4}$                                        & Reduction Type of $E$ at $p=2$   & $C_2$ \\ \hline \hline
1              & any                                                        & $\mathrm{III}$                   & 2     \\ \hline
$\ge 2$      & $a \equiv 1 \pmod{4}$ or $b \equiv 1 \pmod{4}$ (or both) & $\mathrm{I}_k^\ast$ for some $k$ & 2 or 4  \\ \hline
$2$ or $3$ & \multirow{3}{*}{$a \equiv b \equiv 3 \pmod{4}$}            & $\mathrm{III}^\ast$              & 2     \\ \cline{1-1} \cline{3-4}
$4$        &                                                            & $\mathrm{I}_0$ (good)            & 1     \\ \cline{1-1} \cline{3-4}
$5$        &                                                            & $\mathrm{I}_{2m-8}$              & even  \\ \hline
\end{tabular}
\end{center}
\end{proof}


\begin{proposition}\label{prop:commonprime_dividing_a_and_b}
Let $p$ be a prime dividing two of $a$, $b$, or $c$. Then clearly it divides the third. By changing variables in the equation \eqref{eq:Weierstrass_basic_for_type_2_2} via $[p,0,0,0]$ if necessary, we assume $\min \lp \ord_p a, \ord_p b \rp = 1$. Then $E$ has reduction of type $\mathrm{I}_k^\ast$ for some $k$, with even Tamagawa number.
\end{proposition}

\begin{proof}
If $m \neq n$, then we may assume $m > n = 1$ without any loss of generality. By Tate's algorithm, in this case $E$ has reduction of type $\mathrm{I}_k^\ast$ with Tamagawa number 2 or 4. If $m = n = 1$, then we can write $a= pa'$ and $b=pb'$ with $(a',p) = (b',p) = 1$. Hence,
\begin{itemize}
    \item if $a' \not\equiv b' \pmod{p}$, then $E$ has reduction of type $\mathrm{I}_0^\ast$ modulo $p$ with Tamagawa number 4;
    \item if $a' \equiv b' \pmod{p}$, then $E$ has reduction of type $\mathrm{I}_k^\ast$ modulo $p$ with Tamagawa number 2 or 4.
\end{itemize}
\end{proof}



Recall that $E$ is an elliptic curve defined by the equation $y^2 = x(x+a)(x+b)$ with discriminant
$\Delta = 16a^2 b^2 c^2 \neq 0$ where $a, b, c:=a-b \in \ZZ$. We also have assumed that $\min \lp \ord_p a, \ord_p b \rp \le 1$ for all primes $p$. Let
\begin{equation*}
S : = \lbr p \text{ primes}: \ord_p a >0,\, \ord_p b>0 \rbr.
\end{equation*}
If $\# S \ge 2$, then by Proposition
\ref{prop:commonprime_dividing_a_and_b}, then the Tamagawa number
$C$ of $E$ is divisible by 4. Thus the following proposition shows
Theorem \ref{th:torgp_type_2_2}.

\begin{proposition}
Suppose that $\# S \le 1$. Then $4 \mid C$ with only two exceptions: `17a2' and `32a2'. But in both exceptions, we have $C =M =2$.
\end{proposition}

\begin{proof}
Proofs are similar to Proposition \ref{prop:relations_of_S_and_T}.
\end{proof}

\section{$\torgp \simeq \ZZ/4\ZZ$}
\label{section:torgp_type_4}

\begin{theorem}\label{th:torgp_type_4}
If $E$ is an elliptic curve defined over $\QQ$, having rational torsion subgroup $\torgp$ isomorphic to $\ZZ/4\ZZ$, then the order $4 = \# \torgp$ divides $u_K \cdot C \cdot M \cdot \lp \# \Sh(E/K) \rp^{1/2}$.
\end{theorem}

\subsection{Tamagawa numbers}

In order to prove Theorem \ref{th:torgp_type_4}, we first consider Tamagawa numbers of $E$.

From \cite{Ku}, table 3, such elliptic curves can be parametrized by one parameter $\lambda$ by
\begin{equation*}
y^2 + xy - \lambda y = x^3 - \lambda x^2,
\end{equation*}
where the discriminant of the equation $\lambda^4(1+16\lambda) \neq 0$. This is the same as in section \ref{section:torgp_type_2_4}, but without further restriction on $\lambda$. Let $\lambda = \alpha / \beta$, with $\alpha, \beta \in \ZZ$ and $\gcd(\alpha, \beta)=1$. By Proposition \ref{prop:something_good_happened_for_eq_lambda} (a), we may assume $\alpha = 1$. So we begin with the following Weierstrass equation
\begin{equation}\label{eq:type_4_only_beta}
y^2 + \beta xy - \beta^2 y = x^3 - \beta x^2,
\end{equation}
with $\beta \in \ZZ$. Note that this curve has discriminant $\Delta = (16+\beta)\beta^7$ and $c_4 = (16+16\beta+\beta^2)\beta^2$. If $\beta = \pm 1$, then we have either `15a8' or `17a4', both of which have $M=4$. So we may assume that there is at least one prime dividing $\beta$.

Let $p$ be a prime dividing $\beta$, and let $m : = \ord_p \beta > 0$. Write $\beta = p^m u$, for some $u \in \ZZ$ with $\gcd(u,p)=1$. Using Tate's algorithm applied to Weierstrass equations $\displaystyle y^2 + p^{z+1} xy - p^{z+2}u^{-1}y = x^3 -pu^{-1}x^2$ (when $m = 2z+1$ is odd) or $\displaystyle y^2 + p^{z+1} xy - p^{z+2}u^{-1}y = x^3 -pu^{-1}x^2$ (when $m=2z$ is even), we can figure out the reduction types and Tamagawa numbers at primes $p \mid \beta$ for $E$.

\begin{center}
\begin{tabular}{|c|c|c|c|c|}
\hline
$m$                                            & $p$                      & additional conditions             & Reduction Type of $E$ at $p$ & $C_p$ \\ \hline \hline
$m = 2z + 1$ for $z \in \ZZ_{\ge 0}$           & any                      &                                   & $\mathrm{I}_1^\ast$          & 4     \\ \hline
\multirow{3}{*}{$m = 2z$ for $z \in \ZZ_{>0}$} & $p \neq 2$               &                                   & $\mathrm{I}_{2z}$            & even  \\ \cline{2-5}
                                               & \multirow{2}{*}{$p = 2$} & $u \equiv 3 \pmod{4}$ and $m = 8$ & $\mathrm{I}_0$ (good)        & 1     \\ \cline{3-5}
                                               &                          & otherwise                         & bad                          & even  \\ \hline
\end{tabular}
\end{center}
So, in the sequel, we assume
\begin{itemize}
	\item $\ord_p \beta$ is even for all prime $p$;
	\item the number of odd primes dividing $\beta$ is $\le 1$.
\end{itemize}

Moreover, if $\ell$ is an odd prime dividing $\beta + 16$, then $E$ has reduction of type $I_{\ord_\ell (\beta+16)}$ at $\ell$.\footnote{This can be also shown by Tate's algorithm, applied to the equation $\displaystyle y^2 + \beta x y -(\beta+16)^2 y = x^3 - (\beta+ 96) x^2 + 192(\beta+16) x - 128(\beta+24)(\beta+16)$ for $E$.} We furthermore assume throughout this section, that
\begin{itemize}
	\item if $\ell$ is an odd prime dividing $\beta + 16$, then $\ord_\ell \lp \beta + 16 \rp$ is odd. 
\end{itemize}

Suppose that there is no odd prime $p$ dividing $\beta$, i.e., $\beta = \pm 2^m$ for some positive integer $m$. As we can see in the above table, in order to avoid $4 \mid C$, we may assume $m = 2z$ is even. Applying Tate's algorithm to the Weierstrass equation \eqref{eq:type_4_only_beta}, we have the following results.

\begin{center}
\begin{tabular}{|c|c|c|c|}
\hline
$\beta$                       & Curve          & Tamagawa Number             & Manin Constant \\ \hline \hline
$2^2$                         & `40a3'         & $C_2 \cdot C_5 = 2 \cdot 1$ & 2              \\ \hline
$2^4$                         & `32a4'         & $C_2 = 2$                   & 2              \\ \hline
$2^{2z}$ with $z \ge 3$       &                & $C_2=4$                     &                \\ \hline
$-2^2$                        & `24a4'         & $C_2 \cdot C_3 = 2 \cdot 1$ & 2              \\ \hline
$-2^4$                        & singular curve &                             &                \\ \hline
$-2^6$                        & `24a3'         & $C_2 \cdot C_3 = 2 \cdot 1$ & 1              \\ \hline
$-2^8$                        & `15a7'         & $C_3 \cdot C_5 = 1 \cdot 1$ & 2              \\ \hline
$-2^{2z}$ with $z \ge 5$ even &                & $C_2 = 2(z-4)$              &                \\ \hline
$-2^{2z}$ with $z \ge 5$ odd  &                & $C_2 = 2(z-4)$               &                \\ \hline
\end{tabular}
\end{center}
So when $|\beta|$ is a power of 2, then we only need to deal with the cases $\beta = -2^{2z}$ with (i) $z = 4$ or (ii) $z \ge 3$ being odd.

\subsection{$\lp \# \Sh(E/K) \rp^{1/2}$}

In this subsection, we shall see $2 \mid \lp \# \Sh(E/K) \rp^{1/2}$, for various remaining cases left from considerations about Tamagawa numbers. Our main job is to show $\sum i_\ell + \dim \Phi \ge 4$ (notations from subsection \ref{subsection:Kramer}). Then,
\begin{equation*}
\boxed{\sum i_\ell + \dim \Phi \ge 4} \Longrightarrow \boxed{ \dim_{\FF_2} \Sh(E/K)[2] \ge 1} \Longrightarrow \boxed{2 \mid \lp \# \Sh(E/K) \rp^{1/2}}.
\end{equation*}
The first implication follows from Kramer's theorem (see subsection \ref{subsection:Kramer}), and the last implication is due to Kolyvagin's theorem \cite{Kol}.

From the above subsection, we only need to deal with the cases when $\beta$ has at most one odd prime divisor. First, we consider the case where $\beta$ is actually a power of an odd prime.

\begin{proposition}
Suppose that $\beta = p^m$ for some $m >0$. By the previous subsection, we assume 
\begin{itemize}
	\item $m = 2z$ for some positive integer $z$; and
	\item for all odd prime $\ell$ dividing $\beta + 16$, $\ord_\ell \Delta_\mathrm{min} = \ord_\ell \lp \beta + 16 \rp$ is odd.
\end{itemize}
Then we have $\dim_{\FF_2} \Sh(E/K)[2] \ge 1$, i.e., Theorem \ref{th:torgp_type_4} is true, except for a family of curves defined by the equation
\begin{equation*}
y^2 + p^{z} xy - p^{z} y = x^3 - x^2,
\end{equation*}
with $p^{2z} + 16 = \ell^k$ being prime powers.
\end{proposition}

\begin{remark}
For the exceptional family, Theorem \ref{th:torgp_type_4} is also true. This will be shown in the following subsection \ref{subsection:exceptionalZmod4}.
\end{remark}

\begin{proof}
We begin with the following equation: $y^2 + p^{2z} xy - p^{4z} y = x^3 - p^{2z} x^2.$ By a change of variables via $[(1/2)p^{z},0,0,0]$, we get $y^2 + 2p^{z} xy - 8p^{z} y = x^3 - 4x^2.$ Making another change of variables via $[1,4,-p^{z},0]$, we get $y^2 = x^3 + (p^{2z}+8)x^2 + 16x.$ The last equation has discriminant $\Delta = 2^{12}(p^{2z} + 16)p^{2z}$ and $c_4 = 16p^{4z} + 256p^{2z} + 256$. Note that the minimal discriminant of $E$ is given by $\Delta_\mathrm{min} = (p^{2z} + 16)p^{2z}$; in particular, $E$ has good reduction modulo $2$. 

Let $\phi$ be the isogeny $E \to E':= E/E(\QQ)[2]$. (Note that $E(\QQ)[2] \simeq \ZZ/2\ZZ$.) Following \cite{Go}, we compute the Selmer group $\Sel^\phi (E/\QQ)$. For each prime $\ell$ (including $\infty$), we denote by $\delta_\ell$ the map $E'(\QQ_\ell)/ \phi \lp E(\QQ_\ell) \rp \to H^1(\QQ_\ell, E[\phi])$. Since $\Sel^\phi(E/\QQ) \subset H^1(\QQ, E[\phi]) \simeq \sqfr{\QQ}$, the elements of $\Sel^\phi(E/\QQ)$ are those classes of $b \in \QQ^\times$ such that their restrictions $b \in H^1(\QQ_\ell, E[\phi]) \simeq \sqfr{\QQ_\ell}$ are contained in the image $\Ima \delta_\ell$. So by considering the images $\Ima \delta_\ell$, we can figure out which classes are in the Selmer group. For more details of this paragraph, see subsection \ref{subsection:Kramer}.

These local images are given as follows.
\begin{itemize}
    \item $\Ima\delta_\infty = \lbrace 1 \rbrace$.
    \item $\Ima\delta_\ell = \ZZ_\ell^\times \QQ_\ell^{\times 2} / \QQ_\ell^{\times 2}$ for odd primes $\ell \nmid \Delta$.
    \item $\Ima\delta _\ell = \sqfr{\QQ_\ell}$ for odd prime $\ell \mid \Delta$, and $\ell \neq p$.
    \item $\Ima \delta_p = \begin{cases}
    \sqfr{\QQ_p} & \text{ if } p \equiv 1 \pmod{4},\\
    \ZZ_p^\times \QQ_p^{\times 2} / \QQ_p^{\times 2} & \text{ if } p \equiv 3 \pmod{4}.
    \end{cases}$
    \item $\Ima\delta_2 = \lbrace 1, 5 \rbrace \subset \sqfr{\QQ_2}$.
\end{itemize}

Here are some remarks on the odd primes dividing $\Delta$. Since $p^{2z} + 16$ is a sum of two squares, so by the famous theorem on the sum of two squares, if $\ell \equiv 3 \pmod{4}$ divides $\left( p^{2z} + 16 \right)$, then $\ord_\ell \left( p^{2z} + 16 \right)$ must be even. However, we assumed that the exponent $\ord_\ell \left( p^{2z} + 16 \right)$ is always odd. Hence any prime divisor $\ell$ of $\left( p^{2z} + 16 \right)$ must satisfy $\ell \equiv 1 \pmod{4}$.

Let $d$ be a negative, squarefree integer. We now compute the sum of local norm indices $\sum i_\ell$. Note that $i_\infty = 1$. After excluding obvious cases giving $\sum i_\ell \ge 4$, we have the following four cases:
\begin{itemize}
    \item $d = -2$;
    \item $d = -q$ for an odd prime $q$;
    \item $d = -2q$ for an odd prime $q$;
    \item $d = -qq'$ for odd primes $q$, $q'$.
\end{itemize}

Suppose first that $d = -2$. As $\left( \Delta_\mathrm{min}, d \right)_{\QQ_2} = \left( (p^{2z'} + 16)p^{2z}, -2 \right)_{\QQ_2} = \left( 1, -2 \right)_{\QQ_2} = 1$, we have $\sum i_\ell = 3$. Now we compute the Selmer group $\Sel^{\phi_d} (E_d/\QQ)$, where $E_d$ is the quadratic twist of $E$ by $d$, and $\phi_d: E_d \to E_d'$ is the corresponding 2-cyclic isogeny. We denote by $\delta^d_\ell$ the corresponding homomorphism $E_d'(\QQ_\ell) / \phi_d \lp E_d(\QQ_\ell) \rp \to H^1(\QQ_\ell, E_d[\phi_d])$.  Local images $\Ima \delta^d_\ell$ are given as follows.
\begin{itemize}
    \item $\Ima\delta_\infty^d = \sqfr{\RR}$.
    \item $\Ima\delta_\ell^d = \ZZ_\ell^\times \QQ_\ell^{\times 2} / \QQ_\ell^{\times 2}$ for any odd prime $\ell \nmid \Delta$.
    \item $\Ima\delta_\ell^d = \sqfr{\QQ_\ell}$ for any odd prime $\ell \mid \Delta$, with $\ell \neq p$.
    \item $
    \Ima \delta_p^d = \begin{cases}
    \sqfr{\QQ_p} & \text{ if } p \equiv \pm 1 \pmod{8},\\
    \ZZ_p^\times \QQ_p^{\times 2} / \QQ_p^{\times 2} & \text{ if } p \equiv \pm 5 \pmod{8}.
    \end{cases}$
    \item $\Ima\delta_2^d = \lbrace 1, -2 \rbrace$.
\end{itemize}
By Heegner hypothesis, for any prime $\ell \mid \Delta_\mathrm{min}$, we have $\quadsym{-2}{\ell} = \quadsym{-1}{\ell} \quadsym{2}{\ell} = 1$. This implies that such an $\ell$ is congruent to either $1$ or $-5$ modulo $8$. However, by the sum of two squares theorem mentioned above, we must have $\ell \equiv 1 \pmod{8}$ if $\ell \mid \Delta_\mathrm{min}$ and if $\ell \neq p$. If $p \equiv 1 \pmod{8}$, then the image of $p$ is contained in $\Phi$ and is non-trivial, by Proposition \ref{prop:kernel_of_selmer}, and the assumptions we made in the statement of current proposition. So suppose that $p \equiv -5 \pmod{8}$. If there are two distinct odd primes $\ell$ and $\ell'$ dividing $p^{2z} + 16$, then the image of $\ell$ or equivalently of $\ell'$ is contained in $\Phi$ and is non-trivial. So for these cases, we have $\sum i_\ell + \dim_{\FF_2} \Phi \ge 4$. If there is only one odd prime dividing $p^{2z} + 16$, this will be covered in the following subsection \ref{subsection:exceptionalZmod4}.

Suppose that $d = -q$ for some odd prime $q$. Suppose first that $q \equiv 1 \pmod{4}$, i.e. $d \equiv 3 \pmod{4}$. As $\disc(\QQ(\sqrt{d}) |\QQ) = 4d = -4q$, the prime $2$ is ramified in $K = \QQ(\sqrt{d})$. Since $\left( \Delta_\mathrm{min}, d \right)_{\QQ_2} = \left( 1, -q \right)_{\QQ_2} = 1$ as $-q \equiv -1$ or $-5 \pmod{8}$, we have $i_2 = 2$. Since $i_\infty = 1$ and $i_q \ge 1$, we always have $\sum i_\ell \ge 4$.

Now assume that $d = -q$ with a prime $q \equiv 3 \pmod{4}$, then $d \equiv 1 \pmod{4}$. In this case the prime $2$ is unramified in $K$. So we have $i_2 = 0$. Let us consider the Selmer group $\Sel^{\phi_d}(E_d/\QQ)$.
\begin{itemize}
    \item $\Ima \delta_\infty^d = \sqfr{\RR}$.
    \item $\Ima \delta_\ell^d = \ZZ_\ell^\times \QQ_\ell^{\times 2} / \QQ_\ell^{\times 2}$ for any odd prime $\ell \nmid \Delta$, $\ell \neq q$.
    \item $\Ima \delta_\ell^d = \sqfr{\QQ_\ell}$ for any odd prime $\ell \mid \Delta$, and $\ell \neq p$.
    \item $\Ima \delta_p^d = \begin{cases}
    \sqfr{\QQ_p} & \text{ if } p \equiv 1 \pmod{4}, \\
    \ZZ_p \QQ_p^{\times 2} / \QQ_p^{\times 2} & \text{ if } p \equiv 3 \pmod{4}.
    \end{cases}$
    \item $\Ima \delta_q^d = \begin{cases}
    \lbrace 1, qu \rbrace \subset \sqfr{\QQ_q} & \text{ if } q \nmid \left( p^{2z} + 8 \right) \text{ and } \left( \dfrac{p^{2z} + 16}{q} \right) = 1, \\
    \sqfr{\QQ_q} & \text{ otherwise,}
    \end{cases}$ for some $u \in \ZZ_q^\times$.
    \item $\Ima \delta_2^d = \lbrace 1, 5 \rbrace$.
\end{itemize}
Note that for any odd prime $\ell \mid \Delta$, we have $1 = \left( \dfrac{-q}{\ell} \right) = \left( \dfrac{-1}{\ell} \right) \left( -1 \right)^{\frac{q-1}{2}\frac{\ell-1}{2}} \left( \dfrac{\ell}{q} \right) = \left( \dfrac{\ell}{q} \right)$. If $p \equiv 1 \pmod{4}$, then the image of $p$ is contained in $\Phi$ and is non-trivial. Even if $p \equiv 3 \pmod{4}$, if there are at least two odd prime divisors of $\Delta_\mathrm{min}$ apart from $p$, then we also have $\dim_{\FF_2} \Phi \ge 1$, i.e., $\sum i_\ell + \dim_{\FF_2} \Phi \ge 4$. If there is only one odd prime dividing $p^{2z} + 16$, this will be covered in the following `exceptional case' \ref{subsection:exceptionalZmod4}.

Assume $d=-2q$. We have $i_\infty = 1$ always. Note that $\left( \Delta_\mathrm{min}, d \right)_{\QQ_2} = \left( p^{2z} \left( p^{2z} + 16 \right), -2q \right)_{\QQ_2} = \left( 1, -2q \right)_{\QQ_2} = 1$, whence $i_2 = 2$. Since $i_q \ge 1$, we always have $\sum i_\ell \ge 4$.

Finally, assume $d = -qq'$. If the prime $2$ is ramified in $K = \QQ(\sqrt{d})$, then surely we have $\sum_\ell i_\ell \ge 4$. Hence, we must assume the other, i.e., $2$ is unramified, which means that $d \equiv 1 \pmod{4}$. Without loss of generality, we then assume $q \equiv 1 \pmod{4}$ and $q' \equiv 3 \pmod{4}$. Moreover, we further assume $i_q = i_{q'} = 1$, i.e., $ \left( \dfrac{ p^{2z} + 16 }{q} \right) = \left( \dfrac{ p^{2z} + 16 }{q'} \right)= -1$. Now consider the local images of $\Sel^{\phi_d}(E_d / \QQ)$ as follows.
\begin{itemize}
    \item $\Ima \delta_\infty^d = \sqfr{\RR}$.
    \item $\Ima \delta_\ell^d = \ZZ_\ell^\times \QQ_\ell^{\times 2} / \QQ_\ell^{\times 2}$ for any odd prime $\ell \nmid \Delta$, $\ell \neq q$.
    \item $\Ima \delta_\ell^d = \sqfr{\QQ_\ell}$ for any odd prime $\ell \mid \Delta$ and $\ell \neq p$.
    \item $\Ima \delta_p^d = \begin{cases}
    \sqfr{\QQ_p} & \text{ if } p \equiv 1 \pmod{4}, \\
    \ZZ_p \QQ_p^{\times 2} / \QQ_p^{\times 2} & \text{ if } p \equiv 3 \pmod{4}.
    \end{cases}$
    \item $ \Ima \delta_q^d = \begin{cases}
    \sqfr{\QQ_q} & \text{ if } q \nmid \left( p^{2z} + 8 \right), \\
    \lbrace 1, qu \rbrace \subset \sqfr{\QQ_q} & \text{ if } q \mid \left( p^{2z} + 8 \right),
    \end{cases}$    for some $u \in \ZZ_q^\times$.
    \item $\Ima \delta_{q'}^d = \sqfr{\QQ_q}$.
    \item $\Ima \delta_2^d = \lbrace 1, 5 \rbrace$.
\end{itemize}
Note that for any odd primes $\ell$ dividing $\Delta$, we get
\begin{equation*}
1 = \left( \dfrac{-qq'}{\ell} \right) = \left( \dfrac{-1}{\ell} \right) \left( -1 \right)^{\frac{\ell-1}{2}\frac{q-1}{2}} \left( \dfrac{\ell}{q} \right)\left( -1 \right)^{\frac{\ell-1}{2}\frac{q'-1}{2}} \left( \dfrac{\ell}{q'} \right) = \left( \dfrac{\ell}{q} \right) \left( \dfrac{\ell}{q'} \right)
\end{equation*}
and thus we have either $\left( \dfrac{\ell}{q} \right) = \left(\dfrac{\ell}{q'} \right) = 1$ or $\left( \dfrac{\ell}{q} \right) = \left(\dfrac{\ell}{q'} \right) = -1$. Suppose first that $p \equiv 1 \pmod{4}$. If $\left( \dfrac{p}{q} \right) = 1$, then we are done, since the image of $p$ in $\Phi$ is non-trivial. If $\left( \dfrac{p}{q} \right) = -1$, then the image of either $p$ or $pq$ in $\Phi$ is non-trivial. Now, suppose that $p \equiv 3 \pmod{4}$. If there are at least two distinct prime divisors of $p^{2z} + 16$, then among those divisors, at least one $\ell$ must have $\left( \dfrac{\ell}{q} \right) = 1$, since $\left( \dfrac{ p^{2z} + 16 }{q} \right) = -1$. For these cases we have $\sum i_\ell + \dim_{\FF_2} \Phi \ge 4$. If there is only one odd prime dividing $p^{2z} + 16$, this will be covered in the following `exceptional case' \ref{subsection:exceptionalZmod4}.

So far, we have shown that for any cases of $d$, we obtain $\sum i_\ell + \dim_{\FF_2} \Phi \ge 4$ with a family of exceptions. Thus by Kramer's formula, we have $4 \mid C \cdot \lp \# \Sh(E/K) \rp^{1/2}$ for the curves not in the exceptional family.
\end{proof}

\begin{proposition}
Suppose that $\beta = (-1)^s2^m p^{m'}$ for some $s \in \lbr 0, 1 \rbr$ and $m, m' \geq 0$. By the considerations of the above subsection and by the above proposition, the remaining cases are further divided by the following three cases.
\begin{itemize}
\item
$s=1$, $m=2z$ for some $z=3,4$ or  odd $z \geq 5$ and $m'=0$,

\item
$s=1$, $m=0$ and $m'=2z$ for some $z \in \ZZ_{>0}$, or

\item
$s=1$, $m = 8$ and $m' = 2z$ for some odd $z \in \ZZ_{>0}$.
\end{itemize}
Furthermore by the above subsection, we assume for all odd prime $\ell$, $\ord_\ell \lp \beta + 16 \rp$ is either zero or odd. For each such case, we have $\dim_{\FF_2} \Sh(E/K)[2] \ge 1$, i.e., Theorem \ref{th:torgp_type_4} is true.
\end{proposition}

\begin{proof}
Proofs are similar to Proposition 5.2.
\end{proof}

\subsection{Exceptional case}
\label{subsection:exceptionalZmod4}

This family is parametrised by the following Weierstrass equation:
\begin{equation*}
E: y^2 + p^{z} xy - p^{z} y = x^3 - x^2,
\end{equation*}
where $p$ is an odd prime congruent to $3$ modulo $4$. We consider the cases when $p^{2z} + 16$ is an odd power of a prime, in other words, $p^{2z} + 16 = q^k$ for some odd prime $q$ and odd integer $k$. When $k > 1$, such Diophantine equation has only integer solution $p^{z} = 3$, $q = 5$, and $k = 2$, c.f. \cite{CCS}, Lemma 5.5. But this case corresponds to the curve `15a3', having torsion subgroup $\ZZ/2\ZZ \oplus \ZZ/4\ZZ$. So we can exclude it from our consideration, and we may assume $k = 1$, i.e., $p^{2z} + 16$ is a prime.

Here the discriminant $\Delta = p^{2z} \lp p^{2z} + 16 \rp = p^{2z}q$, which is the minimal discriminant. The conductor of the curve $E$ is $pq$, and $E(\QQ)_\mathrm{tors} = \ZZ/4\ZZ$.

Let $G$ be the unique subgroup of $E(\QQ)_\mathrm{tors}$ of order 2, and let $E'$ be the curve $E/G$. We can find a Weierstrass equation for $E'$ thanks to Vélu's formulae (cf. \cite{MMR}). The Weierstrass equation for $E'$ is given as follows:
\begin{equation*}\label{eq:1-Epr}
y^2 + p^z xy - p^z y = x^3 - x^2 -5x - (p^{2z} + 3),
\end{equation*}
with discriminant $\Delta' = p^{4z}(p^{2z} + 16)^2$. Factoring 2-torsion polynomial, we see that $E'$ contains the full 2-torsion subgroup in $E'(\QQ)$: their $x$-coordinates are: $3$, $-1$, and $-\lp p^{2z} + 1 \rp/4$. In particular, the weak Gross--Zagier conjecture is true for $E'$ (cf. \S \ref{section:torgp_type_2_4} and \S \ref{section:torgp_type_2_2}). Now the next corollary follows from the isogeny invariance of the Gross--Zagier conjecture.

\begin{corollary}[to Proposition \ref{prop:isoginv}]
The weak Gross--Zagier conjecture is true for the elliptic curve $E$ in the family and the quadratic field $K$ satisfying Heegner hypothesis.
\end{corollary}

\begin{proof}
Take $\theta: E' \to E$ be the isogeny dual to $E \to E/G$ and take modular parametrisations respecting $\theta$, i.e., we first choose a modular parametrisation $\pi'$ of $E'$ and let $\pi = \theta \circ \pi'$ be the modular parametrisation for $E$. By Proposition \ref{prop:isoginv} (c), and by the remark just below the proposition, for a fixed $E$ in the family, the weak Gross--Zagier conjecture is true except possibly for at most 4 quadratic fields. Since we only concern 2-divisibility, let us try to figure out the quadratic fields satisfying $2 = \ord_2 E'(\QQ)_\mathrm{tors} < \ord_2 E'(K)_\mathrm{tors}$. If this inequality is satisfied, then $E'(K)_\mathrm{tors}$ must contain a point of exact order 4.

The $4$-torsion polynomial for $E'$ (i.e., the polynomial whose roots are the $x$-coordinates of the points in $E'[4](\overline\QQ)$) is given as follows:
\begin{equation*}
f_1(x) f_2(x) f_3(x) g(x)
\end{equation*}
where
\begin{itemize}
\item
$f_1(x) = 2x^2 + (p^{2z} + 4) x + (-p^{2z} + 2)$,

\item
$f_2(x) = x^2 - 6x - (p^{2z} + 7)$,

\item
$f_3(x) = x^2 + 2x + (p^{2z} + 1)$,

\item
$g(x) = \lp 4x + p^{2z} + 4 \rp \lp x + 1 \rp \lp x - 3 \rp$.
\end{itemize}
Evidently, the roots of $g(x)$ correspond to points in $E'[2]$. Discriminants $d_i$ of $f_i(x)$ are as follows:
\begin{itemize}
\item
$d_1 = p^{2z} \lp p^{2z} + 16 \rp = p^{2z} q$,

\item
$d_2 = 4 \lp p^{2z} + 16 \rp = 4q$,

\item
$d_3 = -4p^{2z}$.
\end{itemize}
Thus if $K = \QQ(\sqrt{d})$ is a quadratic field, the polynomials $f_i(x)$ do not have roots in $K$ unless $K = \QQ(\sqrt{-1})$ or $K = \QQ(\sqrt{q})$. Note that $\QQ(\sqrt{q})$ is a real quadratic field, which is not in our concern. If $K = \QQ(\sqrt{-1})$, then we have $u_K = 2$. As we already knew $2 \mid C_E \cdot \lp \# \Sh(E/K) \rp^{1/2}$, we have $4 \mid u_K \cdot C_E \cdot \lp \# \Sh(E/K) \rp^{1/2}$, and the weak Gross--Zagier conjecture is also true for this case.
\end{proof}


\section{$\torgp \simeq \ZZ/2\ZZ$}
\label{section:torgp_type_2}

\begin{theorem}\label{th:torgp_type_2}
If $E$ is an elliptic curve defined over $\QQ$, having rational torsion subgroup $\torgp$ isomorphic to $\ZZ/2\ZZ$, then the order $2 = \# \torgp$ divides $u_K \cdot C \cdot M \cdot \lp \# \Sh(E/K) \rp^{1/2}$.
\end{theorem}

For such elliptic curves, we can find a Weierstrass model following Kubert \cite{Ku}:
\begin{equation}\label{eq:Weierstrass_basic_type_2}
y^2 = x^3 + Ax^2 + Bx,
\end{equation}
where $A, B \in \ZZ$. Note that $A$ and $B$ are not necessarily relatively prime. This elliptic curve has discriminant $\Delta = 16B^2 (A^2 - 4B)$ and $c_4 = 16(A^2 - 3B)$. Let $N$ be the conductor of $E$, and $\Delta_\mathrm{min}$ be the minimal discriminant of $E$.

\subsection{Tamagawa numbers}

The purpose of this subsection is to compute Tamagawa numbers of $E$ at various primes, in order to reduce the cases. Remaining cases will be dealt with in the subsequent subsections. More precisely, we show the following.

\begin{proposition}\label{prop:final_reduction_about_A_and_B}
Let $E$ be an elliptic curve defined by the equation \eqref{eq:Weierstrass_basic_type_2}.
\begin{enumerate}
    \item If $\gcd(A,B) \neq 1$, i.e., if there is a common prime dividing both $A$ and $B$, then $2 \mid C$.
    \item If $B \not\in \lbr 1, -1, 16, -16 \rbr$, then $2 \mid C$.
    \item If $p$ is an odd prime such that $\ord_p \lp A^2 - 4B \rp$ is even, then $2 \mid C_p$.
\end{enumerate}
\end{proposition}

\begin{proof}
(a) Let $p$ be a prime dividing both $A$ and $B$. First of all, if both $\ord_p A \ge 2$ and $\ord_p B \ge 4$ are true, then we can make a change of variables via $[p,0,0,0]$ to get another equation of the same form as equation \eqref{eq:Weierstrass_basic_type_2} with $(A,B)$ replaced by $(A/p^2, B/p^4)$. Consequently, we assume either $\ord_p A <2$ or $\ord_p B <4$. Using Tate's algorithm, we find reduction types for $E$ modulo $p$ as summarised in the following.

\begin{center}
\begin{tabular}{|c|c|c|c|}
\hline
$\ord_p B$         & $\ord_p A$    & Reduction Types of $E$ modulo $p$ & Tamagawa Number $C_p$ \\ \hline\hline
1                  &               & $\mathrm{III}$                    & 2                     \\ \hline
2                  &               & $\mathrm{I}_k^\ast$ for some $k$  & even                  \\ \hline
\multirow{2}{*}{3} & 1             & $\mathrm{I}_k^\ast$ for some $k$  & even                  \\ \cline{2-4}
                   & $\ge 2$       & $\mathrm{III}^\ast$               & 2                     \\ \hline
$\ge 4$            & 1 (bindingly) & $\mathrm{I}_k^\ast$ for some $k$  & even                  \\ \hline
\end{tabular}
\end{center}

(b) Let $p$ be a prime such that $p \mid B$ but $p \nmid A$. Using Tate's algorithm, we have the following results.
\begin{center}
\begin{tabular}{|c|c|c|c|c|}
\hline
                         & $\ord_p B$ & $A \mod{4}$          & Reduction Types of $E$ modulo $p$                      & Tamagawa Number $C_p$ \\ \hline \hline
$p \neq 2$               &            &                      & $\mathrm{I}_n$ with $n = \ord_p \Delta = 2 \ord_p B$   & even                  \\ \hline
\multirow{6}{*}{$p = 2$} & 1          &                      & $\mathrm{III}$                                         & 2                     \\ \cline{2-5}
                         & 2          &                      & $\mathrm{I}_k$ for some $k$                            & even                  \\ \cline{2-5}
                         & $\ge 3$    & $-1$                 & $\mathrm{I}_k$ for some $k$                            & even                  \\ \cline{2-5}
                         & 3          & \multirow{3}{*}{$1$} & $\mathrm{III}^\ast$                                    & even                  \\ \cline{2-2} \cline{4-5}
                         & 4          &                      & $\mathrm{I}_0$ (good)                                  & 1                     \\ \cline{2-2} \cline{4-5}
                         & $\ge 5$    &                      & $\mathrm{I}_n$ with $n= \ord_p \Delta = 2\ord_p B - 8$ & even                  \\ \hline
\end{tabular}
\end{center}
In particular, if $B \not\in \lbr 1, -1, 16, -16 \rbr$, we always have $2 \mid C$.

(c) Let $p$ be an odd prime such that $\ord_p \lp A^2 - 4B \rp$ is an even positive integer. By (a), we assume $p \nmid AB$. Tate's algorithm tells us that in this case, $E$ has reduction of type $\mathrm{I}_{\ord_p \lp A^2 -4B \rp}$, and we have $2 \mid C_p$.
\end{proof}

\begin{proposition}\label{prop:final_reduction_about_A_and_B_2}
Let $E$ be an elliptic curve defined by the equation \eqref{eq:Weierstrass_basic_type_2}.
\begin{enumerate}
    \item Suppose that $B = 1$. If $A \equiv 0$ or $1 \pmod{4}$, then $C_2= 1$. If $A \equiv 3 \pmod{4}$, then $C_2 = 2$. When $A \equiv 2 \pmod{4}$, the situation is more complicated, and we summarise the value $C_2$ modulo $2$ according to $A \pmod{128}$ as follows.
    \begin{center}
\begin{tabular}{|c||c|c|c|c|c|c|c|c|c|c|c|c|c|c|c|c|}
\hline
$A \mod{128}$ & 2  & 6  & 10 & 14 & 18 & 22 & 26 & 30 & 34 & 38  & 42  & 46  & 50  & 54  & 58  & 62  \\ \hline
$C_2 \mod{2}$ & 0  & 0  & 1  & 0  & 0  & 0  & 1  & 1  & 0  & 0   & 1   & 0   & 0   & 0   & 1   & 1   \\ \hline\hline
$A \mod{128}$ & 66 & 70 & 74 & 78 & 82 & 86 & 90 & 94 & 98 & 102 & 106 & 110 & 114 & 118 & 122 & 126 \\ \hline
$C_2 \mod{2}$ & 0  & 0  & 1  & 0  & 0  & 0  & 1  & 1  & 0  & 0   & 1   & 0   & 0   & 0   & 1   &     \\ \hline
\end{tabular}
    \end{center}
    If $A \equiv 126 \pmod{128}$, then the parity of $C_2$ is the same as the parity of $\ord_2 \lp A+2 \rp$. Moreover, $E$ has good reduction modulo 2 if and only if $A \equiv 62 \pmod{128}$. In particular, $C_2$ is odd if and only if $A \equiv 0 \pmod{4}$; $A \equiv 1 \pmod{4}$; $A \equiv 10 \pmod{16}$; $A \equiv 62 \pmod{128}$; or $A \equiv 126 \pmod{128}$ and $\ord_2 \lp A + 2 \rp$ is odd.

    \item Suppose that $B = -1$. Then $C_2$ is even if and only if $A \equiv 0 \pmod{4}$.

\end{enumerate}
\end{proposition}

\begin{proof}
Tate's algorithm.
\end{proof}

\begin{remark}
By Propositions \ref{prop:final_reduction_about_A_and_B} and \ref{prop:final_reduction_about_A_and_B_2}, we assume the following throughout this section; $E$ is an elliptic curve defined by the equation \eqref{eq:Weierstrass_basic_type_2} for relatively prime $A \in \ZZ$ and $B \in \lbr 1, -1, 16, -16 \rbr$ with discriminant $\Delta = 2^4 B^2 (A^2 - 4B)$, such that all odd prime divisors of $A^2 - 4B$ has odd exponent. Moreover,
\begin{itemize}
    \item when $B = 1$, we assume $A \equiv 0 \pmod{4}$, $A \equiv 1 \pmod{4}$, $A \equiv 10 \pmod{16}$, $A \equiv 62 \pmod{128}$, or $A \equiv 126 \pmod{128}$ and $\ord_2 \lp A + 2 \rp$ is odd;
    \item when $B = -1$, we assume $A \not\equiv 0 \pmod{4}$;
    \item when $B = \pm 16$, we assume $A \equiv 1 \pmod{4}$.
\end{itemize}
\end{remark}

\begin{remark}
We furthermore assume $\Delta >0$, by removing finitely many exceptional cases by explicit computation. As $\Delta = 16 B^2 (A^2 - 4B)$, we need to check the cases $(A,B) = (0,1)$, $(\pm 1, 1)$ and $(\pm n, 16)$ for $n = 0, 1, \cdots, 7$. This is easy with Sage Mathematics Software \cite{sagemath}.
\end{remark}

\subsection{$\lp \# \Sh(E/K) \rp^{1/2}$}

In this subsection, we shall see $2 \mid \lp \# \Sh(E/K) \rp^{1/2}$, for various remaining cases left from considerations about Tamagawa numbers. Our main job is to show $\sum i_\ell + \dim \Phi \ge 4$ (notations from subsection \ref{subsection:Kramer}). Then,
\begin{equation*}
\boxed{\sum i_\ell + \dim \Phi \ge 4} \Longrightarrow \boxed{ \dim_{\FF_2} \Sh(E/K)[2] \ge 1} \Longrightarrow \boxed{2 \mid \lp \# \Sh(E/K) \rp^{1/2}}.
\end{equation*}
The first implication follows from Kramer's theorem (see subsection \ref{subsection:Kramer}), and the last implication is due to Kolyvagin's theorem \cite{Kol}.

\begin{proposition}
Suppose that $B = -1$. By the considerations of the subsection above, we assume
\begin{itemize}
	\item $A \equiv 0 \pmod{4}$, and
	\item if $\ell$ is an odd prime dividing $A^2 - 4B = A^2 + 4$, then it has odd exponent.
\end{itemize}
Then we have $\sum i_\ell + \dim_{\FF_2} \Phi \ge 4$, i.e., $2 \mid \lp \# \Sh(E/K) \rp^{1/2}$, except for
\begin{itemize}
    \item `128b2' and `128d2', for which $2 \mid M$;
    \item a family of curves for which $A^2 + 4$ is a power of a prime number.
\end{itemize}
\end{proposition}

\begin{remark}
The exceptional family will be dealt with in the next subsection \ref{subsection:exceptionalZmod2}.	
\end{remark}

\begin{proof}

Our elliptic curve $E$ is given by
\begin{equation}\label{eq:Weierstrass_B=-1}
y^2 = x^3 + Ax^2 - x,
\end{equation}
such that $A \equiv 1, 2, 3 \pmod{4}$. In any cases, the minimal discriminant of the curve becomes $\Delta_\mathrm{min} = \Delta = 2^4(A^2+4)$. In particular, the prime $2$ is always a bad one. Since $2$ must split comletely in $K$, we must have $d \equiv 1 \pmod{8}$. Since $i_q \ge 1$ for each prime divisor $q$ of $d$, we may assume that there are at most 2 prime divisors in $d$, as $i_\infty =1$ always. Glueing this with the fact that $d \equiv 1 \pmod{8}$, we have either $d = -q$ for an odd prime $q$ such that $q \equiv -1 \pmod{8}$ or $d = -qq'$, for distinct odd primes $q$ and $q'$ such that $q \equiv 1 \pmod{4}$ and $q' \equiv 3 \pmod{4}$ with either $(q,q') \equiv (1,-1) \pmod{8}$ or $(q,q') \equiv (5,-5) \pmod{8}$.

If $\ell$ is an odd prime dividing $\Delta$, i.e., $\ell \mid \lp A^2 + 4 \rp$, then by the ``sum of two squares'' theorem, we must have $\ell \equiv 1 \pmod{4}$, i.e., $\ell \equiv 1$ or $5 \pmod{8}$.

Now we compute the group $\Sel^\phi(E/\QQ)$. Note the following local images. Definitions for $\delta_\ell$ are the same as in \S \ref{section:torgp_type_4}.

\begin{itemize}
    \item $\Ima \delta_\infty = \lbrace 1 \rbrace$.
    \item $\Ima \delta_p = \ZZ_p \QQ_p^{\times 2} / \QQ_p^{\times 2}$, for odd primes $p \nmid \Delta$.
    \item $\Ima \delta_p = \sqfr{\QQ_p}$ for odd primes $p \mid \Delta$.
    \item $\Ima \delta_2 = \begin{cases}
    \lbrace 1, 5 \rbrace & \text{ if } A \equiv 1, 3 \pmod{4}, \\
    \lbrace 1, 2, 5, 10 \rbrace & \text{ if } A \equiv 2 \pmod{4}.
    \end{cases}$
\end{itemize}

Suppose that $d = -q$ with $q \equiv -1 \pmod{8}$ Since $\displaystyle \left( \frac{-q}{p} \right) = \left( -1 \right)^\frac{p-1}{2} \left( -1 \right)^{\frac{p-1}{2}\frac{q-1}{2}} \left( \frac{p}{q} \right)$, we have $\displaystyle \left( \frac{p}{q} \right) = 1$ for any odd prime $p \mid (A^2 +4)$. If $A \equiv 1, 3 \pmod{4}$, then $A^2 +4$ is odd, and as every prime divisor of $A^2 +4$ has odd exponent, we can conclude that $\displaystyle \left( \frac{A^2 + 4}{q} \right) = 1$ because the left hand side of the expression is the product of $\displaystyle \left( \frac{p}{q} \right)$ running over all primes $p \mid (A^2+4)$. Secondly, suppose that $A \equiv 2 \pmod{4}$. This means that there is an integer $k$ such that $A = 2+4k$, whence $A^2 + 4 = 16k^2 + 16k +8 = 2^3 (2k^2 + 2k +1)$, so $\ord_2 (A^2 + 4) = 3$. In this case, $\displaystyle \left( \frac{A^2+4}{q} \right) = \left( \frac{2}{q} \right) \prod_{\substack{p \text{ odd primes, } \\ p \mid A^2+4}} \left( \frac{p}{q} \right) = 1$, since $q \equiv -1 \pmod{8}$. Therefore, we always have $\displaystyle \left( \frac{A^2 + 4}{q} \right) = 1$, i.e., $i_q = 2$, whence $\sum i_\ell = 3$.

Now consider the Selmer group $\Sel^{\phi_d}(E_d/\QQ)$. The local images are given as follows.
\begin{itemize}
    \item $\Ima \delta_\infty^d = \lbrace 1 \rbrace$.
    \item $\Ima \delta_p^d = \ZZ_p \QQ_p^{\times 2} / \QQ_p^{\times 2}$, for odd primes $p \nmid \Delta q$.
    \item $\Ima \delta_p^d = \sqfr{\QQ_p}$ for odd primes $p \mid \Delta$.
    \item $\Ima \delta_q^d = \lbrace 1 \rbrace$.
    \item $\Ima \delta_2^d = \begin{cases}
    \lbrace 1, 5 \rbrace & \text{ if } A \equiv 1, 3 \pmod{4}, \\
    \lbrace 1, 2, 5, 10 \rbrace & \text{ if } A \equiv 2 \pmod{4}.
    \end{cases}$
\end{itemize}
If $A \equiv 2 \pmod{4}$ then $\ord_2 (A^2 +4) = 3$. As $2$ is a quadratic residue modulo $q$, and since $A^2 + 4$ must have at least one odd prime except for the cases $A = \pm 2$, the image of $2$ gives a nontrivial element in $\Phi$. When $A = \pm 2$, the curve is equal to `128b2' or `128d2'. In these cases $M = 2$. If $A \equiv 1$ or $3 \pmod{4}$ and if there are at least two prime divisors of $A^2 + 4$, then we can also find a nontrivial element in $\Phi$. If $A^2 +4$ is a power of a prime, then this will be dealt as exceptional cases. See \ref{subsection:exceptionalZmod2}.

Suppose that $d = -qq'$ with $q \equiv 1 \pmod{4}$ and $q' \equiv 3 \pmod{4}$. First note that we are reduced to the case that $\displaystyle \left( \frac{A^2 + 4}{q} \right) = \left( \frac{A^2+4}{q'} \right) = -1$, since otherwise we have $\sum i_\ell \ge 4$. Moreover, if $\ell \mid A$, for $\ell = q$ or $q'$ then we have $\quadsym{A^2 + 4}{\ell} = \quadsym{4}{\ell} = 1$, a contradiction. Now we impose the Heegner hypothesis. At first, since the prime $2$ must split completely in $K$, so thus $d \equiv 1 \pmod{8}$, and we have $(q, q') \equiv (1,  -1)$ or $\equiv (5,-5) \pmod{8}$. For odd primes $p$ dividing $\Delta$, we must have $p \equiv 1 \pmod{4}$ by `sum of two squares theorem', and thus we have to have either $\quadsym{p}{q} = \quadsym{p}{q'} = 1$ or $\quadsym{p}{q} = \quadsym{p}{q'} = -1$.

Local images for the Selmer group $\Sel^{\phi_d}(E_d/\QQ)$ are given as follows:
\begin{itemize}
    \item $\Ima \delta_\infty^d = \lbrace 1 \rbrace$;
    \item $\Ima \delta_p^d = \ZZ_p^\times \QQ_p^{\times 2} / \QQ_p^{\times 2}$, for odd primes $p \nmid \Delta q$ (including $p = q'$);
    \item $\Ima \delta_p^d = \sqfr{\QQ_p}$, for odd primes $p \mid \Delta q$;
    \item $\Ima \delta_2^d = \begin{cases}
    \lbrace 1, 5 \rbrace & \text{ when } A \equiv 1, 3 \pmod{4}, \\
    \lbrace 1, 2, 5, 10 \rbrace & \text{ when } A \equiv 2 \pmod{4}.
    \end{cases}$
\end{itemize}
Similar as above, if $A \equiv 2 \pmod{4}$, then the image of $2$ in $\Phi$ is a non-trivial element, so that $\dim_{\FF_2} \Phi \ge 1$. Now assume $A$ is odd. If $A^2 + 4$ have at least two distinct odd prime divisor, then the image of either one of them gives a non-trivial element in $\Phi$. If $A^2 +4$ is a power of a prime, then this will be dealt as exceptional cases. See \ref{subsection:exceptionalZmod2}.
\end{proof}

\begin{proposition}
Suppose that $B = 1$. By the considerations of the subsection above, we assume
\begin{itemize}
	\item $A \equiv 0 \pmod{4}$, $A \equiv 1 \pmod{4}$, $A \equiv 10 \pmod{16}$, $A \equiv 62 \pmod{128}$, or $A \equiv 126 \pmod{128}$ (in the last case we also assume $\ord_2 (A +2 )$ is odd); and
	\item if $\ell$ is an odd prime dividing $A^2 - 4B = A^2 - 4$, then it has odd exponent.
\end{itemize}
Then we have $\sum i_\ell + \dim_{\FF_2} \Phi \ge 4$, i.e., $2 \mid \lp \# \Sh(E/K) \rp^{1/2}$, except for `17a3', `32a3', and `80a2', for which $2 \mid M$.
\end{proposition}

\begin{proof}

The main difficulty to prove this proposition is due to a large number of cases for $A$. The key to overcome is to group the various cases into 3 categories: (i) $E$ has good reduction modulo $2$, i.e., $A \equiv 62 \pmod{128}$, (ii) $A \equiv 126 \pmod{128}$, and (iii) the remaining cases. After this, proofs are nothing special.
\end{proof}

\begin{proposition}
Suppose that $B = 16$. By the considerations of the subsection above, we assume
\begin{itemize}
	\item $A \equiv 1 \pmod{4}$, and
	\item if $\ell$ is an odd prime dividing $A^2 - 4B = A^2 - 64$, then it has odd exponent.
\end{itemize}
Then we have $\sum i_\ell + \dim_{\FF_2} \Phi \ge 4$, i.e., $2 \mid \lp \# \Sh(E/K) \rp^{1/2}$, except for `17a4', for which $2 \mid M$.
\end{proposition}

\begin{proof}
Proofs are similar as above.
\end{proof}

\begin{proposition}
Suppose that $B = -16$. By the considerations of the subsection above, we assume
\begin{itemize}
	\item $A \equiv 1 \pmod{4}$, and
	\item if $\ell$ is an odd prime dividing $A^2 - 4B = A^2 + 64$, then it has odd exponent.
\end{itemize}
Then we have $\sum i_\ell + \dim_{\FF_2} \Phi \ge 4$, i.e., $2 \mid \lp \# \Sh(E/K) \rp^{1/2}$, except for
\begin{itemize}
    \item $A = 15$, in this case the curve is `272b2' having $C_2 = 2$;
    \item the family characterised by the condition that $A^2 + 64$ is a prime, having $M=2$ for any curve in this family.
\end{itemize}
\end{proposition}

\begin{proof}
Proofs are similar. The exceptional family here is called Neumann--Setzer family, and $M =2$ can be found in \cite{SW04}.
\end{proof}


\subsection{Exceptional case}
\label{subsection:exceptionalZmod2}

In this case we deal with the cases where $A^2 + 4$ is a power of an odd prime. By Lemma 5.4 of Cao--Chu--Shiu \cite{CCS}, then $A^2 + 4$ is a prime unless either $A = 2$ or $A = 11$. For those two non-prime cases, the corresponding curves are `128d2' and `80b4', and both of them have Manin constant 2. Excluding these cases, we assume $A^2 + 4$ is a prime.

This family is parametrised by the following Weierstrass equation: $y^2 = x^3 + Ax^2 - x$, where $A$ is an integer not divisible by $4$, and $A^2 + 4 = p$ is an odd prime. It has discriminant $\Delta=16(A^2+4) = 16p$, which is the minimal discriminant. The conductor of the curve $E$ is $4p$ and $E(\QQ)_\mathrm{tors} = \ZZ/2\ZZ$.

Let $G = E(\QQ)_\mathrm{tors} \simeq \ZZ/2\ZZ$, and consider the curve $E' := E/G$. By Vélu's formula, we can find a Weierstrass equation for $E'$. This is given as follows:
\begin{equation*}
y^2 = x^3 + Ax^2 + 4x + 4A
\end{equation*}
with discriminant $\Delta' = -2^8 (A^2 + 4)^2 = -2^8 p^2$. As the $2$-torsion polynomial for $E'$ is given by
\begin{equation*}
4(x^2 + 4)(x+A),
\end{equation*}
we must have a rational 2-torsion point $P = (-A, 0) \in E'(\QQ)$, i.e., $\ZZ/2\ZZ \subseteq E'(\QQ)_\mathrm{tors}$. If $E'(\QQ)_\mathrm{tors} \supsetneq \ZZ/2\ZZ$, the weak Gross--Zagier conjecture is proved in the above sections. So we assume $E'(\QQ) = \ZZ/2\ZZ$. Making a change of variables $x \mapsto x' = x+ A$, we get another Weierstrass equation
\begin{equation*}
y^2 = x^3 - 2Ax^2 + (A^2 + 4)x.
\end{equation*}
By Tate's algorithm, we know that $C_p = 2$. Thus the weak Gross--Zagier conjecture is true for $E'$ unconditionally.

Now we consider $E'(K)_\mathrm{tors}$ for quadratic field $K$ satisfying Heegner hypothesis. If $\ord_2 E'(K)_\mathrm{tors} > \ord_2 E'(\QQ)_\mathrm{tors}$, then $E'(K)$ must contain either $\ZZ/2\ZZ \oplus \ZZ/2\ZZ$ or $\ZZ/4\ZZ$. For the first case, the 2-torsion polynomial of $E'$ must split into linear factors in $K$. As the polynomial is $4(x^2 + 4)(x+A)$, this happens if and only if $K=\QQ(\sqrt{-1})$. But in this case $u_K = 2$, and the weak conjecture is also true for $E$.

Now suppose that $E(K)_\mathrm{tors} \ge \ZZ/4\ZZ$. By Lemma 13 in \cite{GJT}, we must have $A^2 + 4 = s^2$ for some $s \in \QQ$. But since $A^2 + 4 = p$ is a prime, we cannot have this case. Therefore, we have the following corollary to Proposition \ref{prop:isoginv}.

\begin{corollary}
The weak Gross--Zagier conjecture is true for $E$ in this family and for any quadratic field $K$ satisfying Heegner hypothesis.
\end{corollary}

\part{$E(\QQ)_\mathrm{tors}$ has a point of order 3}

\section{Preliminaries for Part 2}
\label{section:preliminaries_part_2}

\subsection{Optimal curves}
\label{subsection:optimalcurves}

For a positive integer $N$, let $X_1(N)$ and $X_0(N)$ denote the usual modular curves defined over $\QQ$. Let $\mathcal{C}$ denote an isogeny class of elliptic curves defined over $\QQ$ of conductor $N$. For $i=0,1$, there is a unique curve $E_i \in \mathcal{C}$ and a parametrisation $\pi_i: X_i(N) \to E_i$ such that for any $E \in \mathcal{C}$ and parametrisation $\pi_i': X_i(N) \to E$, there is an isogeny $\phi_i:E_i \to E$ such that $\phi_i \circ \pi_i= \pi_i'$.  For $i=0,1$, the curve $E_i$ is called the $X_i(N)$-\emph{optimal curve}.

In \cite{BY}, the authors proved the following theorem, which was conjectured by Stein and Watkins \cite{SW02}.

\begin{theorem}[\cite{BY}, Theorem 1.1]\label{th:Byeon--Yhee}
For $i=0,1$, let $E_i$ be the $X_i(N)$-optimal curve of an isogeny class $\mathcal{C}$ of elliptic curves defined over $\QQ$ of conductor $N$. If there is an elliptic curve $E \in \mathcal{C}$ given by $y^2+axy+y=x^3$ with discriminant $\Delta = a^3-27=(a-3)(a^2+3a+9)$, where $a$ is an integer such that no prime factors of $a-3$ are congruent to $1$ modulo $6$ and $a^2+3a+9$ is a power of a prime number, then $E_0$ and $E_1$ differ by a 3-isogeny, which means that there is an isogeny $\pi: E_0 \to E_1$ with $3 \mid \deg(\pi)$.
\end{theorem}

For any $E \in \mathcal{C}$, we let $E_{\ZZ}$ be the Néron model over $\ZZ$ and $\omega$ be a Néron differential on $E$. Let $\phi:E \to E'$ be an isogeny. We say that $\phi$ is \emph{étale} if the extension $E_{\ZZ} \to E'_{\ZZ}$ to Néron models is étale. If $\phi:E \to E'$ is an isogeny over $\QQ$, then we have $\phi^*(\omega')=n\omega$ for some non-zero integer $n=n_{\phi}$, where $\omega'$ is a  Néron differential on $E'$. The isogeny $\phi$ is étale if and only if $n_{\phi}=\pm 1$. If $\phi:E \to E$ is the multiplication by an integer $m$, then $\phi^*(\omega')=m\omega$. Thus if $\phi$ is any isogeny of degree $p$ for a prime number $p$, we must have $n_{\phi}=1$ or $n_{\phi}=p$. If $\phi'$ denotes the dual isogeny, then $\phi' \circ \phi=[p]$ is the multiplication by $p$ mapping. So precisely one of $\phi$ and $\phi'$ is étale.

In \cite{St}, Stevens proved that in every isogeny class $\mathcal{C}$ of elliptic curves defined over $\QQ$, there exists a unique curve $E_\mathrm{min} \in \mathcal{C}$ such that for every $E \in \mathcal{C}$, there is an étale isogeny $\phi: E_\mathrm{min} \to E$. The curve $E_\mathrm{min}$ is called the \emph{(étale) minimal curve} in $\mathcal{C}$. Stevens conjectured that $E_\mathrm{min}=E_1$ and recently Vatsal proved the following theorem.

\begin{theorem}[\cite{Va}, Theorem 1.10]\label{th:Vatsal}
Suppose that the isogeny class $\mathcal{C}$ consists of semi-stable curves. The étale isogeny $\phi: E_\mathrm{min} \to E_1$ has degree a power of two.
\end{theorem}

\subsection{Cassels' theorem}

Let $F$ be a number field with absolute Galois group $G_F$, $E$ and $E'$ be elliptic curves defined over $F$, and $\phi: E \to E'$ be an isogeny defined over $F$ with dual isogeny $\phi': E' \to E$. For various places $v$ of $F$, let $F_v$ denote the completion of $F$ with respect to the place $v$. The sizes of the $\phi$-Selmer group and the $\phi'$-Selmer group are related by the following theorem of Cassels in \cite{Ca2}.

\begin{theorem}[\cite{Ca2} or \cite{KS}, Theorem 1]\label{th:Cassels}
Suppose $\phi$ is an isogeny from $E$ to $E'$ over $F$. Let $C_{\mathfrak{q}}$ and $C'_{\mathfrak{q}}$ be Tamagawa numbers of $E$ and $E'$ at a finite place $\mathfrak{q}$ of $F$, respectively. Then we have
\begin{equation}
\frac{\# \Sel^{\phi}(E/F)}{\# \Sel^{\phi'}(E'/F) }=
\frac{\# E(F)[\phi] \cdot \prod_{v} \int_{E'(F_{v})} \left|\omega' \right|_{v} \cdot \prod_{\mathfrak{q}} C'_{\mathfrak{q}}}
{\# E'(F)[\phi']  \cdot \prod_{v} \int_{E(F_{v})} \left|\omega \right|_{v} \cdot \prod_{\mathfrak{q}} C_{\mathfrak{q}}},
\end{equation}
where $v$ runs through the infinite places, and $\mathfrak{q}$ runs through the finite places.
\end{theorem}

\section{$\torgp \simeq \ZZ/2\ZZ \oplus \ZZ/6\ZZ$}
\label{section:torgp_type_2_6}

In this section, we prove the Main Theorem for the cases when $\torgp$ is isomorphic to $\ZZ/2\ZZ \oplus \ZZ/6\ZZ$.

\begin{lemma}\label{lem:torsionpoints_generating_component_group}
Let $E$ be an elliptic curve defined over $\QQ$ given by $y^2+a_1xy+a_3y=x^3+a_2x^2+a_4x+a_6$ with $a_i \in \ZZ$, and $P$ be a torsion point of $E(\QQ)$ of a prime order $\ell$. Suppose that $E$ has bad reduction at $p$, having Weierstrass equation of the form $y^2+\overline{a_1} xy =x^3+\overline{a_2}{x}^2$ over $\FF_p$, where $\overline{a_i} = a_i \pmod{p}$. If the point $P$ goes to $(0,0)$ in the reduced curve, then $\ell$ divides $C_p$.
\end{lemma}

\begin{proof}
Let $E_0(\QQ_p)$ be the group of $\QQ_p$-rational points of $E$ which become non-singular points in the reduced curve modulo $p$. Since $P$ becomes singular, the class $P + E_0(\QQ_p) \in E(\QQ_p)/E_0(\QQ_p)$ is non-trivial. Since $[\ell] P = O$, the identity element in $E(\QQ)$, the order of the element $P + E_0(\QQ_p)$ is exactly $\ell$ in $E(\QQ_p)/E_0(\QQ_p)$. Thus $\ell \mid C_p = \left[ E(\QQ_p) : E_0(\QQ_p) \right]$.
\end{proof}

\begin{theorem}
Suppose that $\torgp$ is isomorphic to $\ZZ/2\ZZ \oplus \ZZ/6\ZZ$. Then the order $12 = \# \torgp$ divides the Tamagawa number $C$ of $E$.
\end{theorem}

\begin{proof}
From \cite{Ku}, Table 3, elliptic curves defined over $\QQ$ having torsion subgroup $\ZZ/2\ZZ \oplus \ZZ/6\ZZ$ are parametrized as follows:
\begin{equation}
y^2+ (u-v) xy - uv(v+u) y=x^3 - v(v+u)x^2,
\end{equation}
with $u,v \in \ZZ$, $\gcd(u,v)=1$ and $\displaystyle{\frac{u}{v}=\frac{(T-3S)(T+3S)}{2S(5S-T)}}$ for a pair of relatively prime integers $S,T \in \ZZ$, with $S > 0$. In the expression of $\dfrac{u}{v}$, the numerator $(T-3S)(T+3S)$ and the denominator $2S(5S-T)$ of the right hand side are relatively prime outside $2$. Originally Kubert used a parametrisation seemingly different from the current one, but with some routine computations, readers may see that they are in fact equivalent. The discriminant $\Delta$ of the above equation is given by
\begin{eqnarray*}
\Delta &=& v^6 (v+u)^3u^2(9v+u)\\
&=&2^6S^6(5S-T)^6 (S-T)^6 (T-3S)^2 (T+3S)^2 (9S-T)^2.
\end{eqnarray*}
Let $P=(0,0)$ be a torsion point of $E(\QQ)$ of order $6$. One can easily check that $\Delta$ is minimal at every prime $p \mid v$ because $p$ cannot divide $c_4=(u + 3 v)(u^3 + 9 u^2 v + 3 u v^2 + 3 v^3)$. Similarly, $\Delta$ is minimal at every prime $q \mid (v+u)$ and $r \mid u$ possibly except for $q=2$ and $r=3$.

Suppose either $S$ or $T$ is even. As $\gcd(S,T) = 1$, then the other should be odd. In this case $(T-3S)(T+3S)$ and $2S(5S-T)$ are relatively prime, and thus $u = (T-3S)(T+3S)$ and $v = 2S(5S-T)$. Suppose now that $S-T \neq \pm 1$. In this case there are two distinct primes $p \mid 2S(5S-T)$, (in fact, $p \mid v$) and $q \mid (S-T)$ (in fact, $q \mid (v+u)$ and $q$ is odd). Modulo these primes $p$ and $q$, the curve $E$ has split multiplicative reduction. By \cite{AEC}, Appendix C, Corollary 15.2.1, we have $6 \mid C_p$ and $6 \mid C_q$. So $36 \mid C$.

Now assume $S-T = \pm 1$. In this case we can find two distinct primes $p = 2$ and $q$ dividing $v = 2S(5S-T)$. Similar as above, modulo these primes $p=2$ and $q$, the curve $E$ has split multiplicative reduction. By \cite{AEC}, Appendix C, Corollary 15.2.1, we have $6 \mid C_p$ and $6 \mid C_q$. So $36 \mid C$.

Now we assume that both $S$ and $T$ are odd. If $S=1$, then with the condition $\Delta \neq 0$, there is an odd prime $p \mid (5-T)(1-T)$ (in fact, $p \mid v(v+u)$) and a prime $r \mid (T-3)(T+3)$ (in fact, $r \mid u$ and $r \neq 3$, $p$). Since $E$ has split multiplicative reduction modulo $p$, by \cite{AEC} Appendix C, Corollary 15.2.1, we have $6 \mid C_p$. Since $E$ has bad reduction modulo $r$, where the reduced equation is given by the following form $y^2+\overline{a_1}xy=x^3+\overline{a_2}x^2$ modulo $r$, by applying Lemma \ref{lem:torsionpoints_generating_component_group} to the point $[3]P=(uv, uv^2)$ of order 2, we have $2 \mid C_r$. So $12 \mid C$.

If $S \neq 1$, then there is an odd prime $p \mid S$ (in fact, $p \mid v$) at which $E$ has split multiplicative reduction. By \cite{AEC}, Appendix C, Corollary 15.2.1, we have $6 \mid C_p$. When $(5S-T)(S-T)$ has an odd prime factor $q$ (in fact, $q|v(v+u)$ and  $q \neq p$), then $E$ has split multiplicative reduction modulo $q$. Similarly, we have $6 \mid C_q$. So $36 \mid C$.

Suppose that $|5S-T|=2^A$ and $|S-T|=2^B$. From the condition that
$S$ is odd and $S \neq 1$, one can find that either $A=2$ or $B=2$ (and not both). If $A=2$, we have $T=5S\pm 4$. Substituting this into $(T-3S)(T+3S)$, we can find an odd prime $r \mid (S\pm 2)(2S\pm 1)$ (in fact, $r \mid u$ and $r\neq p$) at which $E$ has bad reduction $y^2+\overline{a_1}xy=x^3+\overline{a_2}x^2$. Note that if $(S\pm 2)(2S\pm 1)$ is a power of $3$, then we must have $S = 1, 2$ or $5$. As the cases $S = 1$ or $2$ are dealt in the above paragraphs, we can choose $r \neq 3$ if $S\neq 5$. Moreover, if $S=5$ and $T=5S-4$, then $\ord_3 \Delta < 12$, so $\Delta$ is also minimal at $r=3$. By applying Lemma \ref{lem:torsionpoints_generating_component_group} for the point $[3]P=(uv, uv^2)$ of order 2, we have $2 \mid C_r$. So $12 \mid C$. If $B=2$, we have $T=S\pm 4$. There is a prime $r \mid (S \mp 2)(S \pm 1)$ (in fact, $r \mid u$  and $r \nmid 3p$) at which $E$ has bad reduction with reduced equation $y^2+\overline{a_1}xy=x^3+\overline{a_2}x^2$. By applying Lemma \ref{lem:torsionpoints_generating_component_group} for the
point $[3]P=(uv, uv^2)$ of order 2, we have $2 \mid C_r$. So $12 \mid C$.
\end{proof}


\section{$\torgp \simeq \ZZ/3\ZZ$}
\label{section:torgp_type_3}

In this section, we prove the Main Theorem for the cases when $\torgp$ is isomorphic to $\ZZ/3\ZZ$. More precisely, we show the following theorem.

\begin{theorem}\label{th:torgp_type_3}
Suppose that $\torgp$ is isomorphic to $\ZZ/3\ZZ$. Then the order $3 = \# \torgp$ divides $u_k \cdot C \cdot M \cdot \lp \# \Sh(E/K) \rp^{1/2}$.
\end{theorem}

\subsection{Tamagawa numbers}

Let $E$ be an elliptic curve defined over $\QQ$ with a rational torsion point of order 3. We can take a Weierstrass equation for $E$ of the following form:\begin{equation}
y^2+axy+by=x^3,
\end{equation}
with $a,b \in \ZZ$, $b >0$, and such that there is no prime number $q$ satisfying both $q \mid a$ and $q^3 \mid b$. The discriminant of $\Delta$ of $E$ is given by $\Delta={b}^3({a}^3-27b)$, which is the minimal discriminant $\Delta_\mathrm{min}$ for $E$. Let $T=\{P:=(0,0), (0,-b), O \}$ be the rational torsion subgroup of order 3.


Suppose first that $b \neq 1$. Then there is a prime $p \mid b$, and Lemma \ref{lem:torsionpoints_generating_component_group} or Tate's algorithm shows $3 \mid C_p$. So we assume $b = 1$ in the sequel.

\subsection{Manin constants}

We introduce an useful theorem by T. Hadano.

\begin{theorem}[\cite{Ha}, Theorem 1.1]\label{th:Hadano}
The quotient curve $E':=E/T$ has a rational point of order 3 if and only if $b$ is a cube $t^3$ with $t>0$. Moreover the curve $E'$ is given by the equation
\begin{equation}
y^2+(a+6t)xy+(a^2+3at+9t^2)ty=x^3
\end{equation}
with discriminant $\Delta' = t^3(a^2+3at+9t^2)^3(a-3t)^3$.
\end{theorem}

Let $E'$ be the curve $E/T$. Since $b=1=1^3$, Theorem 9.2 says that $E'$ also has a rational point of order 3. Thus we have a `chain' $E \to E' \to E''$ of elliptic curves and isogenies of degree 3. Each isogeny in the above chain is étale because its kernel is isomorphic to $\ZZ/3\ZZ$ (the group scheme) since each kernel of the isogenies $E \to E'$ and $E' \to E''$ consists of $\QQ$-rational points of order $3$.

It follows from \cite{Ke}, Proof of Theorem 2, such a chain in the isogeny class of $E$ must have length at most $4$. However, we can readily check that if there is a chain of exact length $4$, then it must be identical to the chain $\text{`27a4'} \to \text{`27a3'} \to \text{`27a1'} \to \text{`27a2'}$, and in this case we can check that $3 \mid M$.

Denote by $\calC$ the rational isogeny class of the curve $E$, and let $E_\mathrm{min}$, $E_0$, and $E_1$ be the étale minimal curve, the $X_0(N)$-optimal curve, and the $X_1(N)$-optimal curve in the isogeny class $\calC$, respectively (cf. see subsection \ref{subsection:optimalcurves}).
Ignoring the above case, we have $E = E_\mathrm{min}$, the étale minimal curve in $\calC$, up to isogeny of degree prime to $3$. Thus, by Vatsal's theorem \ref{th:Vatsal}, we have $E = E_\mathrm{min} = E_1$. Since we have a canonical étale isogeny $E_1 \to E_0$ called the Shimura covering (cf. Remark 1.8 in \cite{Va}) having degree divisible by 3, if $E \neq E_0$ then we have $3 \mid M$. So we assume here and thereafter that $E = E_\mathrm{min} = E_1 = E_0$.
Thus, by (the contrapositive statement of) Theorem \ref{th:Byeon--Yhee}, either there are at least two distinct primes dividing $a^2+3a+9$ or there is a prime $p$ such that $p \mid (a-3)$ and $p \equiv 1 \pmod{3}$.

\subsection{$\lp \# \Sh(E/K) \rp^{1/2}$}

Now Theorem \ref{th:torgp_type_3} is reduced to the following proposition.


\begin{proposition}
Let $E$ be an elliptic curve of conductor $N$ defined by a minimal Weierstrass equation $y^2+ a xy+y=x^3$ with $a \in \ZZ$. Let $K$ be an imaginary quadratic field satisfying the Heegner hypothesis with $\disc(K) \neq -3$, i.e., $u_K \neq 3$ such that $E(K)$ has rank 1. Suppose that either (i) there are at least two distinct primes dividing $a^2+3a+9$; or (ii) there is a prime $p$ such that $p \mid ( a-3) $ and $p \equiv 1 \pmod{3}$. Then $3$ divides $C \cdot \lp \# \Sh(E/K) \rp^{1/2}$.
\end{proposition}

\begin{proof}
Let $\phi$ be an isogeny defined over $\QQ$ of degree 3 from E to the quotient curve $E'$ of E by the torsion subgroup $T=\{P, 2P, O\}$ and $\phi': E' \to E$ be its dual isogeny. Since $\# E(K)[\phi] =3$ and $\# E(K)[\phi']  = 1$, we have $\dfrac{\# E(K)[\phi] }{\# E'(K)[\phi']}=3$. Since $K$ is an imaginary quadratic field, by Theorem 1.2 in \cite{DD}, we have
\begin{equation*}
\prod_{\nu |
\infty}\frac{\int_{E'(K_{\nu})}|\omega'|_{\nu}}{\int_{E(K_{\nu})}|\omega|_{\nu}}
=\frac{\int_{E'(\CC)}|\omega'|}{\int_{E(\CC)}|\omega|}=
3^{-1}|\phi^*(\omega')/\omega| =1 \text{ or } 3^{-1}.
\end{equation*}

Assume that $3 \nmid C$, whence $3 \nmid \prod_\frakq C_\frakq$. By Theorem \ref{th:Cassels}, we have
\begin{equation}\label{eq:lowerbound_of_Sel_pi}
\dim_{\FF_3} \Sel^{\phi}(E/K) \ge \ord_3 \left( \prod_\frakq C'_\frakq \right).
\end{equation}

Suppose that there are at least two distinct primes dividing $a^2+3a+9$. Then at least one of them, say $p$, is not $3$. By Hadano's theorem \ref{th:Hadano} and Lemma \ref{lem:torsionpoints_generating_component_group}, we have $3 \mid C'_{p}=C'_{\mathfrak{p}}=C'_{\overline{\mathfrak{p}}}$ and $3 \mid C'_{q}=C'_{\mathfrak{q}}=C'_{\overline{\mathfrak{q}}}$. Thus from \eqref{eq:lowerbound_of_Sel_pi}, we have $\dim_{\FF_3} \Sel^{\phi}(E/K) \ge 4$.

Suppose that there is a prime $p$ such $p \mid (a-3)$ and $p \equiv 1 \pmod{3}$. Then there is at least one prime $q \neq p$ such that $q \mid (a^2+3a+9)$. By Theorem \ref{th:Hadano} and Lemma \ref{lem:torsionpoints_generating_component_group}, we have $3 \mid C'_{q}=C'_{\mathfrak{q}}=C'_{\bar{\mathfrak{q}}}$. Since $E'$ has split multiplicative reduction at $p$ and $3 \mid \ord_p(\Delta')=-\ord_p(j')$, where $\Delta'$ and $j'$ are the discriminant and the $j$-invariant of $E'$ respectively, we have $3 \mid C'_{p}=C'_{\mathfrak{p}}=C'_{\overline{\mathfrak{p}}}$ by \cite{AEC}, Appendix C, Corollary 15.2.1.
Thus from \eqref{eq:lowerbound_of_Sel_pi}, we have that $\dim_{\FF_3}\Sel^{\phi}(E/K) \geq 4$.

From the following short exact sequence of $G_K$-modules $0 \to E[\phi] \to  E[3] \xrightarrow{\phi} E'[\phi'] \rightarrow 0$, we have the following long exact sequence: $\cdots \rightarrow H^0(G_K, E'[\phi']) \rightarrow H^1(G_K, E[\phi]) \xrightarrow{\imath} H^1(G_K, E[3]) \rightarrow \cdots$. Since $E'(K)[\phi']=0$, $\imath$ is injective and thus $\dim_{\FF_3} \Sel^{3}(E/K) \geq \dim_{\FF_3} \Sel^{\phi}(E/K)$. Thus we conclude that
for the two cases,
\begin{equation}\label{eq:lowerbound_of_Sel_3}
\dim_{\FF_3} \Sel^{3}(E/K) \geq 4.
\end{equation}

From the condition that $E(K)$ has rank 1, we have $E(K)/3E(K) \simeq \ZZ/3\ZZ \oplus \ZZ/3\ZZ$. So the following descent exact sequence $0 \rightarrow E(K)/3E(K) \rightarrow \Sel^{3}(E/K) \rightarrow \Sh(E/K)[3] \rightarrow 0$ and equation \eqref{eq:lowerbound_of_Sel_3} imply that $\dim_{\FF_3} \Sh(E/K)[3] \geq 2$, and therefore $3 \mid \lp \# \Sh(E/K)[3] \rp^{1/2}$.
\end{proof}

\hfill

\noindent Dongho Byeon\\
Department of Mathematical Sciences, Seoul National University,\\
1 Gwanak-ro, Gwanak-gu, Seoul, South Korea,\\
E-mail: \url{dhbyeon@snu.ac.kr}\\

\noindent Taekyung Kim\\
Department of Mathematical Sciences, Seoul National University,\\
1 Gwanak-ro, Gwanak-gu, Seoul, South Korea,\\
E-mail: \url{Taekyung.Kim.Maths@gmail.com}\\

\noindent Donggeon Yhee\\
School of Mathematics and Statistics, University of Sheffield,\\
Room K28, Hicks Building, Hounsfield Road, Sheffield, S3 7RH, United Kingdom,\\
E-mail: \url{dgyhee@gmail.com}

\end{document}